\documentclass[graybox]{svmult}

\usepackage{type1cm}        
\usepackage{makeidx}         
\usepackage{graphicx}        
\usepackage{multicol}        
\usepackage[bottom]{footmisc}

\usepackage{amsfonts}
\usepackage{mathrsfs}

\usepackage{newtxtext}       %
\usepackage[varvw]{newtxmath}       

\usepackage{tikz}

\usepackage{mathdots}

\usetikzlibrary{arrows.meta}

\def\N{\mathbb{N}}
\def\Z{\mathbb{Z}}
\def\R{\mathbb{R}}

\def\B{\mathcal{B}}
\def\F{\mathcal{F}}
\def\G{\mathcal{G}}

\def\PP{\mathcal{P}}

\def\U{\mathcal{U}}
\def\V{\mathcal{V}}
\def\W{\mathcal{W}}

\def\UU{\mathfrak{U}}

\def\s{\mathfrak{s}}

\def\hN{{}^*\N}

\def\hA{{}^*\!A}
\def\hB{{}^*B}
\def\hX{{}^*X}
\def\hf{{}^*\!f}
\def\hg{{}^*g}

\def\ueq{{\,{\sim}_{{}_{\!\!\!\!\! u}}\;}}

\def\fe{\le_{\textrm{fe}}}

\makeatletter
\newcommand{\ostar}{\mathbin{\mathpalette\make@circled\star}}
\newcommand{\make@circled}[2]{%
  \ooalign{$\m@th#1\smallbigcirc{#1}$\cr\hidewidth$\m@th#1#2$\hidewidth\cr}%
}
\newcommand{\smallbigcirc}[1]{%
  \vcenter{\hbox{\scalebox{0.77778}{$\m@th#1\bigcirc$}}}%
}
\makeatother

\makeindex             

\begin{document}

\title*{Hypernatural numbers in
arithmetic Ramsey theory}
\author{Mauro Di Nasso\orcidID{0000-0001-6103-9775} }
\institute{Mauro Di Nasso \at Universit\`a di Pisa, Italy, \email{mauro.di.nasso@unipi.it}}
%
%
\maketitle


\abstract{
The hypernatural numbers $\hN$ of nonstandard analysis
have recently proven to be an effective tool in arithmetic Ramsey theory.
After introducing the fundamental ``nonstandard" notions,
we present several examples to illustrate the use of this technique in practice. 
In particular, we provide brief nonstandard proofs of some recent results 
concerning the monochromaticity of certain families of finite and infinite configurations.}

\section*{Introduction}

In recent years, the use of nonstandard models of natural numbers
has proved useful in the study of problems in 
combinatorics of numbers (see the monograph \cite{dgl} and references therein).
In particular, several applications have been found in 
the area called ``arithmetic Ramsey theory"
which focuses on the existence, for any finite coloring (partition)
of the natural numbers, of monochromatic patterns defined 
by arithmetic operations.

This article aims to introduce the hypernatural numbers $\hN$ of 
nonstandard analysis at a basic level, present their fundamental properties, 
and show a series of applications to problems in arithmetic Ramsey theory. 

After preliminary definitions and properties, the fundamental notions of largeness,
namely thickness, syndeticity, and piecewise syndeticity, are presented 
in the nonstandard settings, where their topological content is made explicit.
We recall the notion of finite embeddability between sets of natural numbers,
and show that it admits a simple and natural characterization in nonstandard terms.

The tools of $u$-equivalence and iterated hyper-extensions are then
introduced, and used to give brief nonstandard proofs
of several recent theorems on the existence or nonexistence of certain 
monochromatic patterns. 
For example, we provide short nonstandard proofs
of the following results:\footnote
{See Theorems \ref{quasi-Pythagoras}, \ref{exponential}, \ref{sumsquotients},
and Corollary \ref{cor-3sumproducts}.}
\begin{itemize}
\item
\emph{There exists a finite coloring of $\N$ such that
no configuration of the form $x, y, x^2+y^2$ is monochromatic.}
\item
\emph{For every finite coloring of $\N$ there exists a monochromatic 
``exponential" configuration of the form $x, y, y^x$.}
\item
\emph{For every finite coloring of $\N$ there exists a monochromatic ``sum and quotient"
configuration $x, y, x+y, \frac{y}{x}$.}
\item
\emph{For every finite coloring of $\N$ there exists a monochromatic
configuration of the form
$x,y,z,x+y,x\cdot z,y\cdot z,x\cdot z+y\cdot z$ where $x<y<z$.}
\end{itemize}

We also introduce ``Ramsey pairs," a nonstandard notion that has proven useful 
in studying the monochromaticity of certain infinite configurations. 
In this regard, we provide brief proofs of the following:\footnote
{See Theorems \ref{differenceRamsey}, \ref{2n+m}, and \ref{n^2+m}.}
\begin{itemize}
\item
\emph{For every finite coloring of $\N$ there exists an increasing sequence $(a_n)_{n\in\N}$ such
that $a_n$ and $a_m-a_n$ are monochromatic for all $n<m$.}
\item
\emph{There exists a finite coloring of $\N$ such that for every increasing sequence
$(a_n)_{n\in\N}$ the elements $a_n, 2a_n+a_m$ are \emph{not} all monochromatic for $n<m$.}
\item
\emph{There exists a finite coloring of $\N$ such that for every increasing sequence
$(a_n)_{n\in\N}$ the elements $a_n, a_n^2+a_m$ are \emph{not} all monochromatic for $n<m$.}
\end{itemize}

The last two results answered negatively to open questions posed in \cite{kmrr}.

\smallskip
There are strong connections between the use of hypernatural numbers
and iterated hyper-extensions, and algebra in the space of ultrafilters over $\N$;
this point will be clarified in this paper. 
However, here we assume no prior knowledge, so also readers that
are unfamiliar with ultrafilters, or simply dislike them, are encouraged to read
all the main parts of this article.

\smallskip
I thank the anonymous referee for carefully reading this article,
and for the valuable comments that helped improve several
aspects of the presentation.

\smallskip
This article is dedicated to Vitaly Bergelson on the occasion of his 75th birthday. 
He is a big fan of ultrafilters; I hope he can also appreciate hypernatural numbers, 
since they are, after all, ultrafilters in disguise.

\section{The hypernatural numbers $\hN$}

As a preliminary remark, I would like to address those researchers 
who are not familiar with nonstandard analysis and who have often
a somewhat diffident approach towards its methods.
The point I would like to emphasize is that, ultimately, hypernatural numbers 
are intuitively easy to understand, since in most respects they closely resemble 
familiar natural numbers.
They are endowed with addition and product operations and a total order, 
such that $(\hN,+,\cdot,<)$ is a discretely ordered commutative semiring that 
extends that of the natural numbers. The usual ``finite" numbers $\N$ are 
just an initial segment of $\hN$, which also contains ``infinite" numbers 
greater than all finite numbers.
Furthermore, hypernatural numbers possess many more properties, 
beyond that of being a semiring, that make them similar to natural numbers.
Indeed, the following general principle holds:
\begin{itemize}
\item
\emph{Extension principle}:\
All subsets $A\subseteq\N^d$ 
and all functions $f:\N^d\to\N$ have ``coherent" hyper-extensions ${}^*A\subseteq\hN^d$
and ${}^*f:\hN^d\to\hN$, respectively, where the set ${}^*A$ is a superset of $A$,
and the function ${}^*f$ is an extension of $f$.
More generally, all mathematical objects $X$ have
``coherent" hyper-extensions $\hX$.

It is assumed that hyper-extensions of numbers are trivial, \emph{i.e.},
${}^*n=n$ for every $n\in\N$, and that $\hN$ is a proper superset of $\N$.
\end{itemize}

We used the word ``coherent" to mean that every ``elementary property"
of the considered objects holds if and only if the same property holds for the corresponding
hyper-extensions. This is the content of the following

\begin{itemize}
\item
\emph{Transfer principle}:\
For every $X_1,\ldots,X_n$ and for
every ``elementary property" $P$:
$$P(X_1,\ldots,X_k)\Longleftrightarrow P(\hX_1,\ldots,\hX_n)$$
\end{itemize}

By ``elementary property" we mean any property that talks about elements
of the considered objects, but not about subsets or higher order objects.\footnote
{For instance, the \emph{well-ordering principle} of the natural numbers is
not an elementary property because it talks about subsets, and in fact
it holds for $\N$ but fails in $\hN$. To see this, observe that if a number $\nu\in\hN$
is infinite then also its predecessor $\nu-1$ is infinite, and hence the set 
of infinite numbers doesn't have a least element.}

Sometimes, we will use the expression ``by backward \emph{transfer}" to mean that
we are using the right-to-left implication ``$\Longleftarrow$" in the equivalence above.

\smallskip
Nonstandard analysis typically focuses on the \emph{continuum} setting,
and its starting point is the set ${}^*\R$ of hyperreal numbers, 
which is an ordered field that properly extends the real line,
and hence contains infinitesimal and infinite numbers.
In more advanced topics one also
considers hyper-extensions of topological spaces, groups, rings,
spaces of functions, metric spaces, and so forth
(see the volume \cite{lw} and references therein
for a general overview of nonstandard analysis).

Here we focus on the \emph{discrete} setting, and
our basic nonstandard framework will be given by the hypernatural numbers $\hN$.

Below are a few basic properties of hypernatural numbers.

\begin{proposition}
\

\begin{enumerate}
\item
The ``finite" natural numbers are an initial segment of the hypernatural numbers, \emph{i.e.}, 
if $\nu\in\hN\setminus\N$ then $\nu>n$ for every $n\in\N$.
\item
Let $A\subseteq\N$. Then $A=\hA\cap\N$.
\item
Let $A\subseteq\N$. Then $A$ if finite if and only if $A=\hA$.
\end{enumerate}
\end{proposition}

\begin{proof}
(1). For every $n\in\N$, consider 
the formula ``$\forall x\in\N\ (x\ne 1\land \ldots\land x\ne n)\to x>n$".
By \emph{transfer}, we obtain that ``$\forall x\in\hN\ (x\ne 1\land \ldots\land x\ne n)\to x>n$", and so,
if $\nu\in\hN\setminus\N$ then $\nu>n$.

(2) is obtained by the following chain of equivalences that hold for every $n\in\N$:
``$n\in A\leftrightarrow {}^*n\in\hA\leftrightarrow n\in \hA$" (recall that ${}^*n=n$ for every $n\in\N$).

\smallskip
(3). Let $A=\{a_1,\ldots,a_k\}$ be finite. By \emph{transfer} from
$$``\forall x\in\N\ (x\in A\leftrightarrow (x=a_1\lor\ldots\lor x=a_k))"$$
we obtain:
``$\forall x\in\hN\ (x\in\hA\leftrightarrow (x=a_1\lor\ldots\lor x=a_k))$."
(Recall that we are assuming ${}^*a=a$ for every $a\in\N$.)

Conversely, given an infinite $A\subseteq\N$, pick a bijective function $f:\N\to A$. 
Then $\hf:\hN\to\hA$ is also bijective.
Since $\hf$ is an extension of $f$, all images $\hf(\nu)$ where
$\nu$ is infinite belong to $\hA\setminus\{f(n)\mid n\in\N\}=\hA\setminus A$, 
and so $A$ is a proper subset of $\hA$.
\end{proof}

Other simple direct consequences of \emph{transfer} are itemized below,
where $A$ and $B$ are arbitrary sets.\footnote
{To prove $(1)$ one applies \emph{transfer} to all formulas ``$a\in A$".
(2) is obtained from the equivalence: ``$(\exists x\ x\in A)\leftrightarrow(\exists x\ x\in\hA)$".
Properties (3), (4), (5), and (6) are proved by considering the formulas:
``$\forall x\ (x\in A\cap B\leftrightarrow (x\in A)\land(x\in B))$",
``$\forall x\ (x\in A\cup B\leftrightarrow (x\in A)\lor(x\in B))$", and
``$\forall x\ (x\in A\setminus B\leftrightarrow (x\in A)\land(x\notin B))$",
``$\forall x\ (x\in A\times B\leftrightarrow (\exists a\in A\ \exists b\in B\ x=(a,b)))$", respectively.
Finally, to prove (7) one applies \emph{transfer} to the formula:
``$\forall X\ (X\in\PP(A)\to(\forall x\in X\ x\in A))$".}
\begin{enumerate}
\item
$\{{}^*a\mid a\in A\}\subseteq\hA$.
\item
\emph{$A\ne\emptyset$ if and only if $\hA\ne\emptyset$.}
\item
${}^*(A\cap B)=\hA\cap\hB$.
\item
${}^*(A\cup B)=\hA\cup\hB$.
\item
${}^*(A\setminus B)=\hA\setminus\hB$.
\item
${}^*(A\times B)=\hA\times\hB$.
\item
${}^*\PP(A)\subseteq\PP(\hA)$.
\end{enumerate}

The last property is particularly relevant, and the elements of ${}^*\PP(A)$
are named \emph{internal subsets} of $\hA$.
For every infinite set $A$ one can show that the inclusion ${}^*\PP(A)\subsetneq\PP(\hA)$
is proper, and hence there exist subsets $X\subseteq\hA$ that are \emph{external},
\emph{i.e.} they are not internal. 

The importance of the internal subsets of $\hA$
is that they share the same elementary property as the subsets of $A$;
\emph{e.g.}, all nonempty internal subsets of $\hN$ have a least element,
and all bounded nonempty internal subsets of $\hN$ have a greatest element.
These two properties are easily proved by \emph{transfer}, using the elementary properties:
\begin{itemize}
\item
\emph{``$\forall X\in\PP(\N)\ (X\ne\emptyset\to X\ \text{has a least element})$"}, and
\item
\emph{``$\forall X\in\PP(\N)\ (X\ne\emptyset\land \exists y\in\N\,(\forall x\in X\ x\le y))\to X\ 
\text{has a greatest element}$"},
\end{itemize}
respectively. We remark that one cannot apply \emph{transfer} to
properties formalized in the form: ``$\forall X\subseteq\N \ldots$" or ``$\exists X\subseteq\N\ldots$", 
because formulas where quantifiers range over subsets are not elementary
(in fact, they are second-order formulas). 

Note that both the set $\hN\setminus\N$ of infinite hypernatural numbers
and the set $\N\subset\hN$ of finite hypernatural numbers are external, 
because the former does not have a least element, and the latter is bounded
but without a greatest element.

A relevant class of internal sets are the hyperfinite ones.
A set $F\subseteq\hN$ is \emph{hyperfinite} if $F$
belongs to the hyper-extension ${}^*\text{Fin}(\N)$ of the family of finite subsets.
Note that every finite $F=\{\nu_1,\ldots,\nu_k\}\subset\hN$ is hyperfinite; this can be easily verified
by applying \emph{transfer} to the property:
``$\forall x_1,\ldots,x_k\in\N\ \exists F\in\text{Fin}(\N)\ F=\{x_1,\ldots,x_k\}$."
However, the class of hyperfinite sets also contain infinite sets;
typical examples are the intervals $[\nu,\mu]=\{x\in\hN\mid \nu\le x\le\mu\}$ 
where $\mu-\nu\in\hN\setminus\N$ is infinite.
We observe that, with respect to elementary properties,
hyperfinite sets behave as finite sets.
For example, every nonempty hyperfinite $F\subset\hN$ has a least and a greatest element;
furthermore, it is possible to define
the \emph{hyper-sum} of its elements $\sum_{\nu\in F}\nu\in\hN$.

The notion of internal set is really useful but also a subtle one, and in the practice
it takes caution to distinguish between internal and external objects.
Since going in depth into this is not needed to develop the topics of this paper, we do not
elaborate here, and refer the interested reader to \cite[Chapter 2]{dgl} and references therein.

\smallskip
Many further properties are obtained as straightforward applications of the
\emph{transfer principle}. For example, a function $f:A\to B$ 
is 1-1, or onto, or bijective, if and only if its hyper-extension ${}^*f:\hA\to\hB$ is 1-1, 
or onto, or bijective, respectively;
and a binary relation $R$ on the set $X$ is an equivalence relation, or a partial order, 
or a linear order, if and only if its hyper-extension ${}^*R$ is an equivalence relation on $\hX$, 
or a partial order on $\hX$, or a linear order on $\hX$, respectively.\footnote
{As usual in set theory, 
here we identify each binary relation $R$ on $X$ with the set of pairs 
$\underline{R}=\{(x,y)\in X\times X\mid xRy\}$ that satisfy it, and then define 
the hyper-extension ${}^*R$ as the binary relation ${}^*\underline{R}\subseteq\hX\times\hX$.} 
Another consequence of \emph{transfer} that is commonly used in the 
practice is the following natural ``coherence" property of hyper-extensions:

\begin{proposition}
Let $\varphi(x,y_1,\ldots,y_k)$ be an elementary property.
For all $A,B_1,\ldots,B_k$ one has that
$${}^*\{x\in A\mid\varphi(x,B_1,\ldots,B_k)\}=\{x\in\hA\mid \varphi(x,\hB_1,\ldots,\hB_k)\}.$$
\end{proposition}

\begin{proof}
Let $C:=\{x\in A\mid\varphi(x,B_1,\ldots,B_k)\}$.
By \emph{transfer} from ``$\forall x\ (x\in C\leftrightarrow (x\in A\land \varphi(x,B_1,\ldots,B_k))$"
we obtain that $\forall x\ (x\in{}^*C\leftrightarrow (x\in\hA\land \varphi(x,\hB_1,\ldots,\hB_k))$".
\end{proof}

For example, let $\text{Exp}:\N^2\to\N$ be the exponential function $(n,m)\mapsto m^n$.
If $A=\{m^n\mid m,n\in\N\}=\{k\in\N\mid\exists m,n\in\N\ \ k=\text{Exp}(m,n)\}$ is the set of powers in $\N$,
then $\hA=\{\lambda\in\hN\mid \exists \mu,\nu\in\hN\ \ \lambda={}^*\text{Exp}(\mu,\nu)\}=
\{\nu^\mu\mid \mu,\nu\in\hN\}$ is the set of powers in $\hN$. 

We do not review here all basic properties of this kind, as they are really natural to conceive
and really easy to prove by using \emph{transfer}. In any case, we will
always explicitly mention when an application of \emph{transfer} is needed to justify our arguments.

\smallskip
Besides \emph{extension} and \emph{transfer}, there is a third principle that
is used in nonstandard analysis, and that originates from 
the \emph{compactness theorem} of first-order logic.
Roughly, the latter states that if one finds elements that realize any finite subcollection of a 
given family of formulas,
then there must also be an element that simultaneously satisfies all of them,
possibly in some larger ``universe."
Informally, one could say that, by logic compactness, every object that 
potentially exists must actually exist in some coherent universe.
A typical example is obtained by considering the family of formulas $\Phi:=\{x>n\mid n\in\N\}$.
It is readily seen that every finite subfamily of $\Phi$ is realized by any
sufficiently large natural number $x$.
Then, by \emph{logic compactness}, there must be an element $\nu$ that simultaneously 
satisfies all formulas in $\Phi$; in fact, such elements $\nu$ are the infinite numbers
in the hyper-extension $\hN$.

The following \emph{enlarging property} can be seen as
a version of logic compactness, conveniently
formulated as a simple intersection property.
There is a limitation on the size of the considered families that can be
any prescribed cardinal; to our purposes the
cardinality of the continuum $\mathfrak{c}=|\R|=|\mathcal{P}(\N)|$ will be large enough.
\begin{itemize}
\item
\emph{$\mathfrak{c}$-Enlarging property}:\
Let $\G=\{A_i\mid i\in I\}$ be a family of sets where $|I|\le\mathfrak{c}$.
If $\G$ has the \emph{finite intersection property}, \emph{i.e.}
if $A_{i_1}\cap\ldots\cap A_{i_n}\ne\emptyset$ for all $A_{i_1},\ldots,A_{i_n}\in\G$,
then the intersection of all hyper-extensions $\bigcap_{i\in I}\hA_i\ne\emptyset$.
\end{itemize}

For instance, the application of logic compactness considered above can be reformulated
as follows. The countable family $\G=:\{(n,+\infty)\mid n\in\N\}$ of subsets of $\N$ has the
finite intersection property; then, by \emph{enlarging}, 
the intersection $\bigcap_{n\in\N}{}^*(n,+\infty)$ is nonempty, and in fact it
equals the set $\hN\setminus\N$ of infinite hypernatural numbers.

\smallskip
At this point, a question naturally arises: ``How do we know that all this is actually possible?” 
For a logician, probably the simplest way is to apply the \emph{compactness theorem} of first-order logic. 
However, there is a simple direct construction of numbers $\hN$ that
satisfy the three principles of \emph{extension}, \emph{transfer},
and \emph{enlarging}, namely an appropriate \emph{ultrapower} of $\N$. 
This is the content of the last Section \ref{sec-ultrapower}.

\smallskip
It is now time to clarify one point. 
In our opinion -- which is not actually a common view in the nonstandard community -- researchers 
who are still in the process of learning the basics of nonstandard analysis
should first practice, and only then examine the ultrapower construction of $\hN$ or ${}^*\R$. 
In fact, while checking every single property within the ultrapower may give 
the impression of greater security and control, it can often be tedious and 
significantly slows down the work; moreover, it is completely unnecessary.
When studying calculus, one usually first becomes familiar with the algebraic properties 
of real numbers and with completeness, before moving on to their construction as Dedekind cuts or 
as equivalence classes of Cauchy sequences of rational numbers. 
Indeed, in practice the properties of calculus are not verified by looking directly at Dedekind cuts; 
all a beginner needs to know is that real numbers are a complete ordered field.
Similarly, all one needs to know to start working with nonstandard methods
are the three principles mentioned above: \emph{extension}, \emph{transfer}, and \emph{enlargement}. 
Of course, at some point a rigorous justification must be provided, and this is done by 
showing the existence of a model, typically an ultrapower. 
But to begin with, one can simply take those three principles as black box assumptions,
in the same way that first-year calculus students take the existence of a complete ordered field for granted.

\smallskip
To help the intuition, one might just think of the hypernatural numbers $\hN$
as a sort of telescopic view of the natural numbers $\N$ where one is also able
to see that, in fact, there are also ``infinite" numbers which come after all the finite ones.
The relevant aspect here is that all familiar notions remain the same: for example,
there are ${}^*$even and ${}^*$odd hypernatural numbers,
there are ${}^*$prime hypernatural numbers, and so forth.
Also, all ``elementary" properties remain the same: for example,
every hypernatural numbers is the sum of at most four square numbers;
for every $\nu\in\hN$ there is a ${}^*$prime number $\xi\in\hN$ with $\nu<\xi<2\nu$;
and so forth.

\smallskip
Before proceeding, a disclaimer is in order. 
Our formulation of the \emph{transfer principle} is not
a fully rigorous one, as it relies on the notion of ``elementary property"
which was not given a precise formulation (this would require the formal notion
of formula of first-order logic).
Besides, more advanced notions such as that of hyperfinite set, 
that are commonly used in the practice of nonstandard methods,
have not been introduced.
However, our intent here is to keep things as simple as possible, and only
present those notions and properties that are actually needed for the purposes of this paper.
The interested reader is referred to other more comprehensive introductions
(see, \emph{e.g.}, \cite[Chapter 2]{dgl} as an intermediate step, and
the volume \cite{lw} for an all-embracing overview of nonstandard methods.)

\smallskip
At this point, we believe the reader can immediately begin to see the 
nonstandard methods in action.
If necessary, they can return to this section whenever they feel the need 
to further formalize the topics presented, or if something is unclear.
In short, they can proceed as if they were learning a new language:
After assimilating the basic rules and words, rather than spending a lot of time 
studying grammar in depth, it can be helpful to immerse yourself in practice 
and try to understand a conversation or pronounce a few initial sentences, 
returning occasionally to review the fundamentals.

\section{Thickness and syndeticity in $\hN$}

In this section we will recall fundamental notions of ``largeness" that are
used in combinatorics of numbers, namely thickness and syndeticity,
and show how the nonstandard setting of the hypernatural numbers $\hN$
formalizes the intuition that they can be seen as topological notions.

\smallskip
In the following, $\N$ will denote the set of positive integers,
and $\N_0=\N\cup\{0\}$ the set of non-negative integers.
By interval $I$ of $\N$ we mean a finite set of the form 
$$[x,y]=\{n\in\N\mid x\le n\le y\}$$
as determined by two elements $x\le y$ of $\N$. Note that we do not consider ``improper" intervals
of the form $[x,+\infty)=\{n\in\N\mid n\ge x\}$, unless explicitly stated.
Similarly, when saying that $I$ is an interval of $\hN$
we mean that $I=[\xi,\eta]=\{\nu\in\hN\mid \xi\le\nu\le\eta\}$ for some 
appropriate elements $\xi\le\eta$ in $\hN$, and we do not
consider improper intervals of the form $[\xi,+\infty)=\{\nu\in\hN\mid\nu\ge\xi\}$.
Note that, since $\N$ is an initial segment of $\hN$, 
every interval of $\N$ is also an interval of $\hN$.

\smallskip
The notion of length for intervals of $\N$ is extended to intervals of $\hN$ in a natural way.

\begin{definition}
The \emph{length} of an interval $I=[\xi,\eta]\subset\hN$ 
is the hypernatural number $|I|:=\eta-\xi+1$.
We say that an interval $I\subset\hN$ is a \emph{finite interval} if
its length $|I|\in\N$ is a finite number, and say that $I$
is an \emph{infinite interval} if the length $|I|\in\hN\setminus\N$ is an infinite number.
\end{definition}

Note that an interval $I\subset\hN$ is an infinite interval if and only if 
it is an infinite set, and so the two uses of the same word ``infinite" are coherent here.

\smallskip
Let us now recall the following basic notions of largeness for sets of natural numbers.

\begin{definition}
Let $A\subseteq\N$.
\begin{itemize}
\item
$A$ is \emph{thick} if it includes arbitrarily long intervals.
\item
$A$ is \emph{syndetic} if there exists $k$ such that every interval of length $k$ meets $A$,
\emph{i.e.}, $A\ne\emptyset$ and consecutive elements of $A$ have distance at most $k$.
\item
$A$ is \emph{piecewise syndetic} if $A=T\cap S$ is the intersection of a thick set $T$
with a piecewise syndetic set $S$.
\item
$A$ is \emph{syndetically thick} if for every $n$, the set of initial points of intervals
of length $n$ that are included in $A$ is syndetic.\footnote
{In the literature, the name ``thickly syndetic" is often used to refer to these sets.
We think this may be misleading, because it suggests that 
syndeticity is realized in a ``thick" manner, \emph{i.e.}, that syndeticity is realized on a thick set;
but this is actually the notion of piecewise syndeticity.
On the other hand, ``syndetically thick'' suggests that thickness
is realized in a ``syndetic" manner, and this description seems to fit the given definition better.}
\end{itemize}
\end{definition}

It is easily seen that thickness and syndeticity are dual notions, \emph{i.e.},
a set $A$ is thick if and only if the complement $A^c:=\N\setminus A$ is not syndetic.
Similarly, also piecewise syndeticity and syndetic thickness are dual notions,
\emph{i.e.}, $A$ is piecewise syndetic if and only if $A^c$ is not syndetically thick.

In a straightforward way, one proves that the above definitions admit
several equivalent reformulations.
For our purposes, the relevant ones are the following:

\begin{proposition}
Let $A\subseteq\N$. 
\begin{enumerate}
\item
$A$ is piecewise syndetic if and only if there exists $k$
such that for every $n$ there exists an interval $I$ of length $n$ 
such that for every subinterval $J\subseteq I$ of length $k$ one has $A\cap J\ne\emptyset$.
\item
$A$ is syndetically thick if and only if for every $n$ there exists $k$
such that for every interval $I$ of length $k$ there exists a point $x\in I$
such that the interval $[x,x+n)\subseteq A$.
\end{enumerate}
\end{proposition}

The nonstandard characterizations below seem to
exactly match the intuition behind the considered notions.

\begin{theorem}\label{thickness-syndeticity}
Let $A\subseteq\N$.
\begin{enumerate}
\item
$A$ is thick if and only if there exists an infinite interval $I\subseteq\hA$
if and only if for every $\nu\in\hN$ there exists an interval $I\subseteq\hA$ of length $\nu$.
\item
$A$ is syndetic if and only if 
$\hA$ has only finite gaps, \emph{i.e.}, $\hA\cap I\ne\emptyset$ for every infinite interval $I$,
if and only if $\hA\ne\emptyset$
and consecutive elements of $\hA$ have finite distance.
\item
$A$ is piecewise syndetic if and only if there exists an infinite interval $I\subseteq\hN$
such that $\hA\cap I$ has only finite gaps, \emph{i.e.},
$\hA\cap J\ne\emptyset$ for every infinite sub-interval $J\subseteq I$,
if and only if $\hA\cap I\ne\emptyset$ and consecutive elements in $\hA\cap I$ have finite distance.
\item
$A$ is syndetically thick if and only if for every infinite interval $J\subset\hN$
there exists an infinite subinterval $J\subseteq I$ such that $J\subseteq\hA$.
\end{enumerate}
\end{theorem}

\begin{proof}
$(1)$. Let $\varphi(n,A,\N)$ be the elementary property
stating that $A$ includes an interval of length $n$:
$$\varphi(n,A,\N):\ ``\exists x\in\N\ (x+1\in A \land\ \ldots\ \land x+n\in A)."$$

We observe that $A$ is thick if and only if 
``$\forall n\in\N\ \varphi(n,A,\N)$" if and only if (by \emph{transfer})
``$\forall \nu\in\hN\ \varphi(n,\hA,\hN)$" if and only if
for every $\nu\in\hN$ there exists an interval $I\subseteq\hA$ of length $\nu$.
Trivially, this latter property implies the existence of an infinite interval $I\subseteq\hA$,
which in turn implies ``$\varphi(n,\hA,\hN)$" for every $n\in\N$, 
and so, by backward \emph{transfer}, ``$\varphi(n,A,\N)$" for every $n\in\N$.

\smallskip
(2). It directly follows from $(1)$ by recalling that $A$ is syndetic if and only
if the complement $A^c$ is not thick, and that ${}^*(A^c)={}^*(\N\setminus A)=
\hN\setminus\hA=(\hA)^c$.

\smallskip
(3). For $n,k\in\N$ consider the elementary property $\psi(n,k,A,\N)$ 
stating that there exists an interval $I\subset\N$ of length $n\ge k$ such that
every subinterval $J\subseteq I$ of length $k$ meets $A$.\footnote
{A possible formalization is the following:
$$\psi(n,k,A,\N):\ ``n\ge k\land 
\exists\,x\in\N\ \left(\bigwedge_{i=0}^{n-k}\,\bigvee_{j=0}^{k-1}\ x+i+j\in A\right)."$$}
If $A$ is piecewise syndetic then there exists $k\in\N$ such that
``$\forall n\in\N$ $\psi(n,k,A,\N)$" and hence, by \emph{transfer},
``$\forall \nu\in\hN\ \psi(\nu,k,\hA,\hN)$".
This implies that there exists an infinite interval $I\subset\hN$ 
(actually, we can take $I$ of any prescribed infinite length)
such that every subinterval $J\subseteq I$ of length $k$ intersects $\hA$. 
Clearly, this latter property in turn implies that $\hA\cap J\ne\emptyset$ 
for every infinite subinterval $J\subseteq I$.

Conversely, suppose that there exists an infinite interval $I\subset\hN$
such that $\hA\cap I$ has only finite gaps.
Consider the elementary property $\vartheta(x,y,k,A)$
stating that $k$ is the size of the longest
gap of $A$ in the interval $[x,y]$, \emph{i.e.},
$k$ is the greatest length of a subinterval $J\subset[x,y]$ such $J\cap A=\emptyset$
(we allow for the empty interval of length $0$).\footnote
{~A possible formalization is the following:
$$\vartheta(x,y,k,A):\ ``\exists z\in\N\ \left(x\le z\le y-k-1\land \bigwedge_{i=0}^{k-1} z+i\notin A\right)\land
\forall w\in\N\ \left(\bigvee_{j=0}^k w+j\in A\right)."$$}
Since nonempty finite subsets of the integers always have a greatest element,
clearly the property ``$\forall x,y\in\N\ (y>x)\to \exists k\in\N_0\ \vartheta(x,y,k,A)$" holds.
Then, by \emph{transfer}, we obtain that 
``$\forall\xi,\eta\in\hN\ (\eta>\xi)\to \exists \nu\in\hN_0\ \vartheta(\xi,\eta,\nu,\hA)$",
and in particular there exists $\nu\in\hN$ which is the greatest length of
a subinterval $J\subseteq I$ disjoint from $\hA$.
Now observe that by the hypothesis, such a $\nu$ is necessarily finite, 
say $\nu=k\in\N$.
Finally, for every $n\in\N$, the infinite interval $I$ is a witness 
of the elementary property $\psi(n,k,\hA,\hN)$ considered above.
Then, by backward \emph{transfer}, we have that for every $n\in\N$
the property $\psi(n,k,A,\N)$ holds, \emph{i.e.},
$A$ is piecewise syndetic.

\smallskip
(4). It directly follows from the nonstandard characterization (3), by recalling
that a set $A\subseteq\N$ is syndetically thick if and only if its complement
$A^c$ is not piecewise syndetic.
\end{proof}

\begin{example}\label{example-digit1}
Let $B\subseteq\N$ be the set of natural numbers 
whose expansion in base $3$ contains at least one digit $1$.
Then $B$ is syndetically thick.

To see this, let $I=[\alpha,\alpha+L)$ be any infinite interval.
Pick an infinite $M$ such that $3^{M+1}\ll L$, \emph{i.e.}, $\frac{L}{3^{M+1}}$ is infinite.\footnote
{For example, if $M$ is the integer part of $\frac{\log_3L}{2}$,
then $3^{M+1}\ll 3^{2M}\le L$.}
Then at least one (in fact, infinitely many) of the consecutive intervals 
$J_\nu:=[\nu\cdot 3^{M+1}, (\nu+1)\cdot3^{M+1})$ is
included in $I$. We observe that every number in the infinite subinterval
$J'_\nu:=[\nu\cdot 3^{M+1}+3^M, \nu\cdot 3^{M+1}+2\cdot 3^M)\subset J_\nu\subset I$ 
has ternary digit $1$ in place $M$
(counting positions from right to left, starting from $0$).
This shows that $\hB\cap I\supseteq J'_\nu$ includes an infinite interval.\footnote
{With a little effort, one could refine the argument, and show that for every infinite interval $I$
of length $L$ there exists an infinite interval $J\subseteq\hB\cap I$ of length $N$
where $\frac{N}{L}\approx \frac{1}{3}$.}
\end{example}

\subsection{Topological characterizations}\label{sec-thicksynd}

For simplicity, in the following we denote $\N_\infty:=\hN\setminus\N$ the
set of infinite hypernatural numbers.

Since the hypernatural numbers are linearly ordered,
there is a natural topology on $\N_\infty$ that one may consider.

\begin{definition}
A set $X\subseteq\N_\infty$ is \emph{$I$-open} if for every
$\xi\in X$ there exists an infinite $\nu\in\hN$ such that
the interval $[\xi-\nu,\xi+\nu]\subseteq X$.
The \emph{interval topology} $\tau_I$ on $\N_\infty$ is the family of  $I$-open sets.\footnote
{The interval topology $\tau_I$ corresponds to the $U$-topology as introduced
by H.J. Keisler and S.C. Leth in \cite[\S 2]{kl},
where the considered cut $U=\N$ is the one given by the finite numbers.}
\end{definition}

Note that the family of $I$-open sets actually is a topology;
indeed, it is easily verified that arbitrary unions and finite intersections of $I$-open sets 
are $I$-open, that the whole set $\N_\infty$ of infinite numbers is $I$-open, and that
(trivially) the empty set is $I$-open.

\begin{remark}
The interval topology $\tau_I$ is not Hausdorff; precisely,
two points $\mu<\nu$ of $\N_\infty$ are topologically indistinguishable, 
\emph{i.e.}, $\mu\in U\Leftrightarrow\nu\in U$ for every $I$-open set $U$,
if and only if they have finite distance $\nu-\mu\in\N$.
So, if $\approx$ is the equivalence relation of having finite distance,
\emph{i.e.}, $\nu\approx\mu\Leftrightarrow |\mu-\nu|\in\N$,
then the topological quotient $\N_\infty/\approx$ is Hausdorff.
We observe that the quotient topology on
$\N_\infty/\approx$ is in fact the order topology determined by
the order on $\N_\infty/\approx$ as induced by the order on $\hN$.
\end{remark}

Note that if $I=[\xi,\eta]\subseteq\N_\infty$ is an interval, then its interior part 
is the set of its points that are at infinite distance from both end-points, \emph{i.e.}, 
$$\text{int}(I)=\{\zeta\in[\xi,\eta]\mid \zeta-\xi,\eta-\zeta\in\N_\infty\}.$$

\begin{proposition}
If a set $U\subseteq\N_\infty$ is $I$-open and nonempty then there
exists an infinite interval $I\subset\N_\infty$ such that $\text{int}(I)\subsetneq U$.
In consequence, $\mathcal{I}:=\{\text{int}(I)\mid I\subset\N_\infty\ \text{interval}\}$ is a base
of the interval topology $\tau_I$.
\end{proposition}

\begin{proof}
If $U$ is $I$-open and nonempty, pick $\xi\in U$.
Then, by definition of $I$-open set, there exists an infinite $\nu$
such that the infinite interval $I=[\xi-\nu,\xi+\nu]\subseteq U$,
and hence $\text{int}(I)\subsetneq U$, since $\text{int}(I)\subsetneq I$.
For the last statement of the theorem,
one only needs to observe that also the empty open set is obtained,
because  $\text{int}(I)=\emptyset$ for every finite interval $I\subset\N_\infty$.
\end{proof}

\smallskip
Thanks to the interval topology, we can see that the largeness notions 
of thickness and syndeticity are in fact topological notions when seen in the
remote realm of the infinite hypernatural numbers.

\begin{theorem}
For every $A\subseteq\N$ let $A_\infty:=\hA\setminus A=\hA\cap\N_\infty$
be the set of infinite points of $\hA$. The following equivalences hold
for the interval topology $\tau_I$ on $\N_\infty$:\footnote
{Let $Y$ be a subset of a topological space $X$. Recall the following notions:
\begin{itemize}
\item
$Y$ is \emph{dense} if its closure $\overline{Y}=X$.
\item
The \emph{interior part} $\text{int}(Y)$ is the union of all open sets $U\subseteq Y$, or equivalently,
$\text{int}(Y)=(\overline{Y^c})^c$ is the complement of the closure of the complement of $Y$.
\item
$Y$ is \emph{nowhere dense} if for every nonempty open set $U$ there exists 
a nonempty open set $V\subseteq U$ such that $V\cap Y=\emptyset$, or equivalently,
if $\text{int}(\overline{Y})=\emptyset$. 
\item
$Y$ is \emph{somewhere dense} if it is not nowhere dense.
\item
$Y$ is \emph{co-nowhere dense} if the complement $Y^c$ is nowhere dense.
\end{itemize}}
\begin{enumerate}
\item
$A$ is syndetic if and only if $A_\infty$ is dense.
\item
$A$ is thick if and only if the interior part $\text{int}(A_\infty)\ne\emptyset$.
\item
$A$ is piecewise syndetic if and only if $A_\infty$ is somewhere dense.
\item
$A$ is syndetically thick if and only if $A_\infty$ is co-nowhere dense.
\end{enumerate}
\end{theorem}

\begin{proof}
All four properties are direct consequences of the nonstandard characterizations 
of Theorem \ref{thickness-syndeticity}, together with the fact that the interior parts 
$\text{int}(I)$ of the intervals $I\subset\N_\infty$ form a basis of $I$-open sets.
As an example, let us look in detail at the proof of piecewise syndeticity.

\smallskip
(3). Suppose first that $A_\infty$ is somewhere dense, and pick a nonempty open set $U$ 
such that $A_\infty\cap V\ne\emptyset$ for every nonempty open $V\subseteq U$.
Pick an infinite interval $I=[\xi,\eta]\subset\N_\infty$ such that $\text{int}(I)\subseteq U$,
pick an infinite $\nu<|I|/3$. Then $I':=[\xi+\nu,\eta-\nu]\subset\N_\infty$ is an infinite 
interval with $I'\subseteq\text{int}(I)\subseteq U$.
We observe that for every infinite sub-interval $J\subseteq I'$
one has $\hA\cap J\ne\emptyset$. Indeed, $\text{int}(J)$ is a nonempty open subset of $U$
and so, by the hypothesis, $A_\infty\cap\text{int}(J)\ne\emptyset$, and hence
also $\hA\cap J\ne\emptyset$. 

Conversely, if $A$ is piecewise syndetic, pick
an infinite interval $I$ such that for every infinite sub-interval $J\subseteq I$ 
one has $\hA\cap J\ne\emptyset$, and let $U:=\text{int}(I)$.
Then $U$ is open and nonempty; besides, $A_\infty\cap V\ne\emptyset$ for every
nonempty open $V\subseteq U$. To see this, pick an infinite interval $J=[\xi,\eta]\subset\N_\infty$
such that $\text{int}(J)\subseteq V$. Similarly as above, pick an infinite $\nu<|J|/3$;
then $J':=[\xi+\nu,\eta-\nu]\subset\N_\infty$ is an infinite interval with $J'\subseteq\text{int}(J)\subseteq V$.
By the hypothesis, the intersection $\hA\cap J'\ne\emptyset$.
Finally, note that, since $J'\subset\N_\infty$, one has
$\hA\cap J'=A_\infty\cap J'\subseteq A_\infty\cap V$.
\end{proof}

\begin{example}
Let $C\subseteq\N$ be the Cantor ``ternary set" of natural numbers 
whose expansion in base $3$ does not contains any digit $1$.
Then $C_\infty$ is nowhere dense.
Indeed, we have seen in Example \ref{example-digit1} that the 
complement $B=C^c$ is syndetically thick,
and hence $B_\infty=(C_\infty)^c$ is co-nowhere dense.
\end{example}

We note that, as a consequence of the previous theorem, 
the families of syndetic and of piecewise syndetic sets
satisfy all the general properties that are satisfied by
the families of dense and of somewhere dense sets, respectively.\footnote
{~In fact, the embedding $A\mapsto A_\infty$ respect the set theoretic
operations of arbitrary unions and intersections, and of complements, \emph{i.e.},
$(\bigcup_{i\in I}A_i)_\infty=\bigcup_{i\in I}(A_i)_\infty$, 
$(\bigcap_{i\in I}A_i)_\infty=\bigcap_{i\in I}(A_i)_\infty$, and 
$(A\setminus B)_\infty=A_\infty\setminus B_\infty$.}
As a particular case, we obtain an alternative topological proof of the strong partition
regularity of piecewise syndetic sets.

\begin{corollary}
If the set $A\subseteq\N$ is piecewise syndetic and $A=A_1\cup\ldots\cup A_n$ then
one of the pieces $A_i$ is piecewise syndetic.
\end{corollary}

\begin{proof}
If every set $A_i$ were not piecewise syndetic then we would have a union
$A_\infty=(A_1)_\infty\cup\ldots\cup(A_n)_\infty$ where $A_\infty$ is somewhere dense
and every $(A_i)_\infty$ is nowhere dense, contradicting the topological fact that
a finite union of nowhere dense sets is nowhere dense.
\end{proof}

\section{Finite embeddability}

\begin{definition}
Let $A,B\subseteq\N$. 
We say that $B$ is \emph{finitely embeddable} in $A$,
and write $B\fe A$, if every finite subset $F\subseteq B$ has a rightward shift
$x+F\subseteq A$, \emph{i.e.} if for all $b_1,\ldots,b_k\in B$
there exists $x\in\N$ such that $x+b_i\in A$ for every $i=1,\ldots,k$.\footnote
{This notion of embeddability was isolated and studied in \cite{dn};
however, it was already implicitly used in other previous papers in additive number theory.}
\end{definition}

Examples showing that finite embeddability is not reflexive are easy to find
(note that only nontrivial shifts by $x\ne 0$ are allowed).
However, it is readily seen that 
$\fe$ is a transitive relation: ``If $C\fe B$ and $B\fe A$ then $C\fe A$."

It is a simple exercise to verify that finite embeddability is preserved
by taking ``Delta-sets" $\Delta(A):=\{a'-a\mid a<a'\ \text{in}\ A\}$, and also
by taking intersections of shifts.
\begin{itemize}
\item
\emph{If $B\fe A$ then $\Delta(B)\fe\Delta(A)$.}
\item
\emph{If $B\fe A$ then 
$\bigcap_{x\in F}(B-x)\fe\bigcap_{x\in F}(A-x)$ for every nonempty finite $F\subset\N_0$.}
\end{itemize}

\begin{definition}\label{def-remoterealm}
For $A\subseteq\N$ and $\nu\in\hN$, the \emph{remote realm} 
of $A$ at $\nu$ is following set of natural numbers:
$$A_\nu:=\{n\in\N\mid \nu+n\in\hA\}=(\hA-\nu)\cap\N.$$
\end{definition}

Informally, we can say that $A_\nu$ is the set of natural numbers 
that lie in $\hA$ when seen starting from the ``remote" viewpoint $\nu$.

\begin{theorem}\label{finitelyembedded}
Let $A,B\subseteq\N$. The following are equivalent:
\begin{enumerate}
\item
$B\fe A$.
\item
There exists $\nu\in\N$ such that $B\subseteq A_\nu$.
\end{enumerate}
\end{theorem}

\begin{proof}
Suppose first that $B\fe A$, and for every $b\in B$,
let $\Lambda_b:=\{x\in\N_0\mid x+b\in A\}$. By the hypothesis we know that
the family $\{\Lambda_b\mid b\in B\}$ has the finite intersection property
and so, by \emph{enlargement}, we can pick an element $\nu\in\bigcap_{b\in B}{}^*\Lambda_b$.
Since ${}^*\Lambda_b=\{x\in\hN_0\mid x+b\in\hA\}$, this means that
$\nu+B\subseteq\hA$, \emph{i.e.}, $B\subseteq A_\nu$.

Conversely, for all $b_1,\ldots,b_k\in B$, the number $\nu\in\hN$ is a witness
of the following elementary property: ``$\exists x\in\hN\ (x+b_1\in\hA\land\ldots\land x+b_k\in\hA)$."
By backward \emph{transfer} we obtain that ``$\exists x\in\N\ (x+b_1\in A\land\ldots\land x+b_k\in A)$",
as desired.
\end{proof}

The following properties are directly proved from the definitions
and the equivalences of Theorem \ref{thickness-syndeticity}.\footnote
{See \cite[\S 4]{dn}.}

\begin{itemize}
\item
\emph{$A$ is thick if and only if there exists $\nu\in\hN$ such that $A_\nu=\N$
if and only if $\N\fe A$.}
\item
\emph{$A$ is syndetic if and only if $A_\nu\ne\emptyset$ for every $\nu\in\hN$
if and only if $A_\nu$ is syndetic for every $\nu\in\N$.}
\item
\emph{$A$ is piecewise syndetic if and only if there exists $\nu\in\hN$ such that $A_\nu$ is syndetic
if and only if there exists a syndetic set $B$ such that $B\fe A$.}
\item
\emph{$A$ is ``AP-rich", \emph{i.e.}, it includes arbitrarily long arithmetic progressions,
if and only if there exists $\nu\in\hN$ such that $A_\nu$ is AP-rich.}
\end{itemize}
 
Recall the \emph{upper} and \emph{lower asympotic densities} of sets $A\subseteq\N$:
$$\underline{d}(A)=\liminf_{n\to\infty}\frac{|A\cap\{1,\ldots,n\}|}{n};\quad
\overline{d}(A)=\limsup_{n\to\infty}\frac{|A\cap\{1,\ldots,n\}|}{n}$$

Banach density ``refines" the usual asymptotic density
by considering arbitrary intervals $\{x+1,\ldots,x+n\}$ instead of just initial intervals $\{1,\ldots,n\}$.
Precisely, for $A\subseteq\N$, the \emph{lower Banach density} $\underline{\text{BD}}(A)$
and the \emph{upper Banach density} $\overline{\text{BD}}(A)$ are defined as follows:
\begin{multline*}
\underline{\text{BD}}(A):=
\lim_{n\to\infty}\left(\min_{x\in\N}\frac{|A\cap\{x+1,\ldots,x+n\}|}{n}\right)
\\
\overline{\text{BD}}(A):=
\lim_{n\to\infty}\left(\min_{x\in\N}\frac{|A\cap\{x+1,\ldots,x+n\}|}{n}\right).
\end{multline*}

It is easily verified that:
$$\underline{\text{BD}}(A)\le\underline{d}(A)\le\overline{d}(A)\le\overline{\text{BD}}(A).$$

Another notion of density that is widely used in number theory 
is \emph{Schnirelmann density}:
$$\sigma(A):=\inf_{n\in\N}\frac{|A\cap\{1,\ldots,n\}|}{n}.$$

We remark that this density is much stricter than Banach density, and
indeed one has:
\begin{itemize}
\item
$\sigma(A)\le\underline{\text{BD}}(A)$.
\item
\emph{There exist sets $A$ with $\underline{\text{BD}}(A)=1$ and $\sigma(A)=0$.}
\end{itemize}

However, R. Jin \cite{ji} proved that one of the ``remote realms" of a given set 
has a Schnirelmann density equal to the upper Banach density
(see also \cite[Theorem II]{dn} for a self-contained proof).
In standard terms, we have the following property.

\begin{theorem}[Jin]
Let $A\subseteq\N$ have positive Banach density.
Then there exists a set $B\fe A$ such that the Schnirelmann
density $\sigma(B)=\overline{\text{BD}}(A)$.
\end{theorem}

As consequences of this theorem one can obtain Banach density versions
of several results about Schnirelmann density.
For example, the following version of Mann's theorem holds 
(see \cite[Theorem 2]{ji}):\footnote
{Recall Mann's Theorem: \emph{Let $A,B\subseteq\N\cup\{0\}$ be such that $0\in A\cap B$. Then
$\sigma(A+B)\ge\min\{\sigma(A)+\sigma(B),1\}$}.}
\begin{itemize}
\item
\emph{Let $A,B\subseteq\N$. Then
$\overline{\text{BD}}((A+B)\cup(A+B+1))\ge\min\{\overline{\text{BD}}(A)+\overline{\text{BD}}(B),1\}$.}
\end{itemize}

Another example is the version of Plunnecke's inequality
for the Banach density (see \cite{ji2}), parallel to the original one
that was formulated for the Schnirelmann density.

Jin's Theorem also has consequences about the existence of finite
configurations in sets of positive density.

\begin{proposition}
Let $\emptyset\ne\mathfrak{P}\subseteq\PP(\N)$ be a translation invariant family of finite patterns.
If every set of positive Schnirelmann density includes a pattern $P\in\mathfrak{P}$, then
also every set of positive Banach density includes a pattern $P\in\mathfrak{P}$.
\end{proposition}

\begin{proof}
Given $A\subseteq\N$ of positive upper Banach density, pick $\nu$ such that 
$\sigma(A_\nu)=\overline{\text{BD}}(A)$.
By the hypothesis, there exists $Q\in\mathfrak{P}$ such that
$Q\subseteq A_\nu$. Now recall that $A_\nu\fe A$; so, there exists
a shift $x+Q\subseteq A$. Since $\mathfrak{P}$ is shift invariant, 
we have that $P:=x+Q\in\mathfrak{P}$ is the desired pattern included in $A$.
\end{proof}

For example, a proof that every set of positive Schnirelmann densitiy
includes arbitrarily long arithmetic progressions would already be a proof 
of Szemer\'edi's theorem.
Although this observation could potentially be useful, I know of no explicit example 
of this kind where the stronger assumption of a positive Schnirelmann density has been used.

\section{The $u$-equivalence relation on $\hN$}

In this section we introduce the $u$-equivalence relation on hypernatural numbers, 
which is of central importance in partition regularity problems.

\begin{definition}
Two elements $\nu,\mu\in\hN$ are \emph{$u$-equivalent},
and we write $\nu\ueq\mu$, when $\nu\in\hA$ if and only if $\mu\in\hA$ for all $A\subseteq\N$.
\end{definition}

\begin{remark}\label{hNtopologicalspace}
In Section \ref{sec-thicksynd} we have seen the \emph{interval topology}.
However, the topology that is most commonly considered on the hypernatural 
numbers $\hN$ is the \emph{standard topology} $\tau_s$, that has the 
family $\mathcal{S}$ of hyper-extensions as a base of (cl)open sets:\footnote
{The name ``standard topology" originates from the fact that in the 
early years of nonstandard analysis, hyper-extensions $\hA$ were called 
``standard sets" of the nonstandard universe.}
$$\mathcal{S}=\{\hA\mid A\subseteq\N\}.$$ 
The topological space $(\hN,\tau_S)$ is a completely regular zero-dimensional 
compact space where $\N$ is a discrete dense subspace. 
By \emph{enlargement}, it is shown that $(\hN,\tau_S)$ is not Hausdorff, 
in fact not even $T_0$. We observe that the $u$-equivalence on $\hN$
coincide with the Kolmogorov equivalence $\equiv_K$
that holds between points that are topologically indistinguishable.\footnote
{Two points $x,y$ are topologically indistinguishable when
they have the same set of neighborhoods.}
We remark that $\hN$ can be seen
as an ``expansion" of the Stone-\v{Cech} compactification $\beta\N$;
indeed, it is not difficult to show that $\beta\N$ is homeomorphic to the Kolmogorov
quotient space $\hN/\equiv_K$.\footnote
{The interested reader can find further information about the topological space $(\hN,\tau_S)$ 
in \cite{dp}, where dynamics with the use of hypernatural numbers is studied.}
\end{remark}

The following fact establishes a
close connection between $\ueq$-equivalence
and partition regularity problems. 

\smallskip
\begin{theorem}\label{PRcharacterizations}
Let $\mathfrak{P}\subseteq\mathcal{P}(\N)$ be a family of ``patterns" or ``configurations".
Then the following are equivalent:
\begin{enumerate}
\item
$\mathfrak{P}$ is partition regular, \emph{i.e.},
in every finite coloring 
$\mathbb{N}=C_1\cup\ldots\cup C_r$ there exists a 
pattern $P\in\mathfrak{P}$ which is monochromatic, \emph{i.e.},
$P\subseteq C_i$ for a color $C_i$.
\item
There exists $Q\in{}^*\mathfrak{P}$ whose elements are $u$-equivalent to each other.
\item
There exists $\nu\in\hN$ that satisfies the property:
If $\nu\in\hA$ then $P\subseteq A$ for some $P\in\mathfrak{P}$.
\end{enumerate}
\end{theorem}

\begin{proof}
$(1)\Rightarrow(2)$. 
For every $A\subseteq\N$, let 
$\Lambda(A):=\{P\in\mathfrak{P}\mid P\subseteq A\ \text{or}\ P\cap A=\emptyset\}$.
We observe that the family 
$\mathcal{G}:=\{\Lambda(A)\mid A\subseteq\N\}$ has the finite intersection property.
Indeed, given $A_1,\ldots,A_n\subseteq\N$, pick a finite coloring
$\N=C_1\cup\ldots\cup C_r$ such that for every $i=1,\ldots,n$ and for
every $j=1,\ldots,r$ one has that either $C_j\subseteq A_i$ or $C_j\cap A_i=\emptyset$.
By the hypothesis, there exists $P\in\mathfrak{P}$ and a color $C_j$ such that $P\subseteq C_j$;
then clearly $P\in\bigcap_{i=1}^n\Lambda(A_i)$.
By \emph{enlargement} there exists $Q\in\bigcap_{A\in\U}{}^*\Lambda(A)$.
Since $Q\subseteq\hA$ or $Q\subseteq \hA^c$ for every $A\subseteq\N$,
the elements of $Q\in{}^*\mathfrak{P}$ are $u$-equivalent to each other.

$(2)\Rightarrow(3)$.
We show that any $\nu\in Q$ satisfies the desired property.
Let $A\subseteq\N$ be such that $\nu\in\hA$.
Since the elements of $Q$ are $u$-equivalent to each other, then it is $Q\subseteq\hA$.
Finally, by applying \emph{backward transfer} to the following elementary property:
``$(\exists Q\in{}^*\mathfrak{P}) (Q\subseteq\hA)$",
one obtains the existence of an element $P\in\mathfrak{P}$ such that $P\subseteq A$.

$(3)\Rightarrow(1)$.
Let $\N=C_1\cup\ldots\cup C_r$ be a finite coloring.
Pick the color $C_i$ such that $\nu\in{}^*C_i$;
then there exists $P\in\mathfrak{P}$ with $P\subseteq C_i$.
\end{proof}

\begin{remark}
For simplicity, we proved the above characterizations for families of patterns in $\N$, 
but the same result also holds for families $\mathfrak{P}\subseteq\mathcal{P}(X)$
of patterns over an arbitrary set $X$. In this case the PR-witness ultrafilter will be an ultrafilter
on $X$, and the $u$-equivalence relation will be the one defined on $\hX$ by letting
$\xi\ueq\eta$ if and only if $\xi\in\hA\Leftrightarrow\eta\in\hA$ for all $A\subseteq X$.
\end{remark} 

For example, recall the following classic results in arithmetic Ramsey theory:
\begin{itemize}
\item
\emph{Schur's Theorem}: For every finite coloring of the natural numbers
there exists a triple of the form $a, b, a+b$ which is monochromatic.
In other words, the family $\PP=\{\{a<b<a+b\}\mid a<b\}$ is partition regular.
\item
\emph{Van der Waerden's Theorem}: For every finite coloring of the natural numbers
and for every $\ell$ there exists a monochromatic arithmetic progression of length $\ell$.
This means that for every $\ell$ the family
$\PP=\{\{a, a+d,\ldots, a+(\ell-1)d\}\mid a,d\in\N\}$ is partition regular.
\item
\emph{Folkman's Theorem}: For every finite coloring of the natural numbers 
and for every $\ell$, there exist numbers $a_1<\ldots<a_\ell$ such that
all sums of distinct elements are monochromatic. 
This means that the family $\PP=\{\text{FS}(a_1,\ldots,a_\ell)\mid a_1<\ldots<a_\ell\}$
is partition regular, where
$\text{FS}(a_1,\ldots,a_\ell):=\{a_{n_1}+\ldots+a_{n_k}\mid n_1<\ldots<n_k\le\ell\}$.\footnote
{Clearly, Schur's Theorem is the particular case of Folkman's Theorem when $\ell=2$.}
\end{itemize}

By using the nonstandard characterizations, we have that:

\begin{itemize}
\item
\emph{Schur's Theorem} states the existence of hypernatural numbers
$\xi<\eta$ such that $\xi\ueq\eta\ueq\xi+\eta$.
\item
\emph{Van der Waerden's Theorem} states the existence, for any given $\ell\in\N$,
of hypernatural numbers $\alpha,\delta$ such that
$\alpha\ueq\alpha+\delta\ueq\alpha+2\delta\ueq\ldots\ueq\alpha+(\ell-1)\delta$.
\item
\emph{Folkman's Theorem} states the existence, for every $\ell\in\N$,
of hypernatural numbers $\alpha_1<\ldots<\alpha_\ell$ such that
the finite sums $\alpha_{i_1}+\ldots+\alpha_{i_k}$ for $i_1<\ldots<i_k\le\ell$ are
$\ueq$-equivalent to each other.
\end{itemize}

Since the last two properties hold for arbitrarily large $\ell\in\N$, 
an ``overspill" phenomenon occurs, so they can be strengthened as follows.\footnote
{~In fact, \emph{overspill} (or \emph{overflow}) is a general principle in nonstandard analysis according to
which, if an elementary property holds for all finite $n\in\N$, then it also holds
for all $\nu\le L$ for a suitable infinite $L\in\hN$.}

\begin{itemize}
\item
\emph{Van der Waerden's Theorem} states the existence
of hypernatural numbers $\alpha,\delta$ such that
$\alpha\ueq\alpha+\ell\delta$ for every $\ell\in\N$.
\item
\emph{Folkman's Theorem} states the existence
of a sequence $(\alpha_i\mid i\in\N)$ of hypernatural numbers 
such that the finite sums $\alpha_{i_1}+\ldots+\alpha_{i_k}$ for $i_1<\ldots<i_k$ are
$\ueq$-equivalent to each other.
\end{itemize}

\begin{remark}
We warn the reader that, although apparently similar,
the above nonstandard formulation of Folkman's Theorem
is substantially different from Hindman's Theorem.\footnote
{See Theorem \ref{hindman}.}
In fact, the nonstandard version of the latter states the existence
of an increasing internal sequence $(\alpha_\nu\mid\nu\in\hN)$ indexed over the
hypernatural numbers, such that all ``hyperfinite" sums $\sum_{\nu\in F}\alpha_\nu$ 
for nonempty $F\in{}^*\text{Fin}(\N)$ are $\ueq$-equivalent to each other.
Clearly, this is a stronger property than the nonstandard Folkman's Theorem.
\end{remark}

The following properties of the equivalence relation $\ueq$ are useful tools in practice, 
as we will show later with some examples.

\begin{theorem}\label{u-equivalenceproperties}
Let $\nu,\mu\in\hN$, and let $f,g:\N\to\N$. Then
\begin{enumerate}
\item
If $\nu\ueq n$ where $n\in\N$ then $\nu=n$.
\item
If $\nu\ueq\mu$ and $\nu\ne\mu$ then $\nu$ and $\mu$ are at infinite distance.
\item
If $\nu\ueq\mu$ then $\hf(\nu)\ueq\hf(\mu)$.
\item
If $\nu\ueq\mu$ and $f$ has finite range, then $\hf(\nu)=\hf(\mu)$.
\item
If $\hf(\nu)\ueq\nu$ then $\hf(\nu)=\nu$.
\item
If $\hf(\nu)\ueq\hg(\nu)$ and $f$ is ``bounded-to-one", \emph{i.e.},
there exists $k$ such that the fibers $|f^{-1}(\{n\})|\le k$ for every $n$,
then $\hf(\nu)=\hg(\nu)$. 
\end{enumerate}
\end{theorem}

\begin{proof}
All proofs of the above properties are found in \cite[\S 2]{dnultra}, with the only exception 
of the last item, which however, can be deduced with similar arguments.
For completeness, let us see this in detail. 

$(6)$. For $n\in\N$, let $s_n:=|f^{-1}(\{n\})|\le k$, and let
$f^{-1}(\{n\})=\{a_{n,1}<\ldots<a_{n,s_n}\}$.
For $i=1,\ldots,k$ let $C_i:=\{a_{n,i}\mid n\in\N\}$, and consider the partition
$\N=C_1\cup\ldots\cup C_k$. Pick $C_i$ such that $\nu\in{}^*C_i$.
We observe that the restriction $h:f|_{C_i}:C_i\to f(C_i)$ is a bijection.
Note that $\hg(\nu)\ueq\hf(\nu)$ implies that also $\hg(\nu)\in{}^*(f(C_i))$, and we have
$\nu={}^*h^{-1}(\hf(\nu))\ueq {}^*h^{-1}(\hg(\nu))$. By property (5), it follows that
${}^*h^{-1}(\hg(\nu))=\nu$, and hence $\hg(\nu)={}^*h(\nu)=\hf(\nu)$, as desired.
\end{proof}

To illustrate how $u$-equivalence can be used in arithmetic Ramsey theory problems,
let us start with the following simple negative property.

Recall that a Diophantine equation $E(X_1,\ldots,X_k)=0$ is called partition regular
if the family of its solutions $\mathfrak{P}=\{\{n_1,\ldots,n_k\}\mid E(n_1,\ldots,n_k)=0\}$
is partition regular.

\begin{theorem}\label{quasi-Pythagoras}
The equation $X^2+Y^2=Z$ is not partition regular on $\N$.\footnote
{This is a particular case of the following fact proved in \cite{dr}:
``If $k\notin\{n,m\}$ then $X^n+Y^m=Z^k$ is not partition regular on $\N$."}
\end{theorem}

\begin{proof}
By contradiction, suppose there exist $\alpha \ueq \beta \ueq \gamma$ such that
$\alpha^2+\beta^2=\gamma$. Note that 
$\alpha,\beta,\gamma$ are even numbers, since
they cannot all be odd. Then we can write
$$\alpha=2^{a}\alpha _{1},\quad \beta =2^{b}\beta _{1},\quad 
\gamma=2^{c}\gamma _{1}$$
where $a,b,c\ge 1$ and $\alpha_1, \beta_1, \gamma_1$ are odd.
By using property (3) of the previous theorem, 
we see that $a\ueq b\ueq c$ and $\alpha_1\ueq\beta_1\ueq\gamma_1$. 
We now distinguish two cases.

\emph{Case 1:} $a\ne b$. Let us assume that $a<b$ (the other case is symmetric).
Note that $2^{2a}(\alpha_1^2+2^{2b-2a}\beta_1^2)=2^{c}\gamma_1$. 
Since $\alpha_1^2+2^{2b-2a}\beta_1^2$ and $\gamma_1$ are odd, it follows
that $2a=c$. Since $c\ueq a$ we have that $2a\ueq a$ and so,
by property (5) of the previous theorem, we must conclude that $2a=a$. This is 
a contradiction because $a\ne 0$. 

\emph{Case 2:} $a=b$. We have $2^{2a}(\alpha_1^2+\beta_1^2)=2^{c}\gamma_1$. 
Since $\alpha_1,\beta_1$ are odd, $\alpha_1^2+\beta_1^2\equiv 2\mod 4$,
and so $2^c\gamma_1=2^{2a+1}\alpha_2$ where $\alpha_2$ is odd.
But then $2a+1=c\ueq a$ and hence it must be $2a+1=a$, a contradiction.
\end{proof}

\section{Iterated hyper-extensions}

In order to find ``positive" results, $u$-equivalence is usually combined with the technique of
iterated hyper-extensions. We remark that, from a foundational point of view, these iterations
are possible because one can construct a star map $*:\mathbb{V}\to\mathbb{V}$ 
from a ``universe" $\mathbb{V}$ into itself.\footnote
{While this can be done with some little effort (see \cite{dj}
and references therein), it is worth remarking that the existence of star maps $*:\mathbb{V}\to\mathbb{V}$ 
where $\mathbb{V}$ is the whole set-theoretic universe raises interesting foundational questions, that have been studied
by the so-called \emph{nonstandard set theories}.
In fact, to construct such universal star maps,
one needs to drop the axiom of regularity, but still maintain
the existence of a Mostowski transitive collapse of
extensional structures. The resulting axiomatic theory can be shown to be equiconsistent with $\text{ZFC}$.}
Such iterations allow for a precise formalization of the intuition of 
``different levels of infinity." Let us see how.

Since $\hN$ is an end-extension of $\N$, \emph{i.e.},
$x<y$ for all $x\in\N$ and $y\in\hN\setminus\N$, by \emph{transfer}
we obtain that the double hyper-extension ${}^{**}\N$ is an end-extension of
$\hN$, the triple hyper-extension ${}^{***}\N$ is an end-extension of
${}^{**}\N$, and so forth.
So, if a number $\alpha\in\hN\setminus\N$ is infinite then
${}^*\alpha\in{}^{**}\N\setminus\hN$, and hence $\alpha<{}^*\alpha$.
Then, by \emph{transfer}, it follows that ${}^*\alpha<{}^{**}\alpha$, and so forth.
This naturally produces an increasing chain of ``different levels of infinity" 
that has proven instrumental in applications.
$$\alpha\ <\ {}^*\alpha\ <\ {}^{**}\alpha\ <\ {}^{***}\alpha\ <\  \ldots$$

The first level of infinity is given by the infinite hypernatural numbers,
that satisfy the following property:
\begin{itemize}
\item
\emph{Let $X\subseteq\N$ and let $\alpha\in\hN\setminus\N$ be infinite.
If $\alpha\in{}^*X$ then the set $X$ is infinite.}
\end{itemize}

The proof is simple. Assume for the sake of contradiction that $X=\{n_1,\ldots,n_k\}\subset\N$ is finite.
Then by \emph{transfer} from:
``$n_1\in X\land\ldots\land n_k\in X\land\ \forall x\in X\ (x=n_1\lor\ldots\lor x=n_k)$"
we obtain that ${}^*X=\{{}^*n_1,\ldots,{}^*n_k\}$.
Since ${}^*n=n$ for every $n\in\N$, we conclude that $\alpha=n_i$ for some $n_i$, a contradiction.

The second level of infinity is given by hyper-extensions ${}^*\alpha$
of infinite numbers $\alpha\in\hN\setminus\N$. 

\begin{theorem}\label{RamseyNS}
Let $X\subseteq\N^2$, and let $\alpha\in\hN\setminus\N$ be infinite.
If $(\alpha,{}^*\alpha)\in{}^{**}X$ then there exists an infinite set $H$ such that 
the set of pairs $[H]^2:=\{(h,h')\mid h<h'\ \text{in} \ H\}\subseteq X$.
\end{theorem}

\begin{proof}
For every $h\in\N$ denote by 
$X_h:=\{n\in\N\mid (h,n)\in X\}$ the $h$-fiber of $X$.
Let $\Lambda:=\{n\in\N\mid (n,\alpha)\in\hX\}$, and observe that its hyper-extension is the set
${}^*\Lambda=\{x\in\hN\mid (x,{}^*\alpha)\in{}^{**}X\}$.
By the hypothesis $\alpha\in{}^*\Lambda$ and hence $\Lambda$ is infinite.
Pick $h_1\in\Lambda$. Then $(h_1,\alpha)\in\hX$, and so $\alpha\in{}^*(X_{h_1})$.
Since $\alpha\in{}^*(\Lambda\cap X_{h_1})$ we can pick $h_2>h_1$ 
in $\Lambda\cap X_{h_1}$. In particular, $(h_1,h_2)\in X$.
Since $(h_2,\alpha)\in\hX$, we have that $\alpha\in{}^*(X_{h_2})$,
and hence $\alpha\in{}^*(\Lambda\cap X_{h_1}\cap X_{h_2})$.
Then we can pick $h_3>h_2$ in the set $\Lambda\cap X_{h_1}\cap X_{h_2}$.
In particular, $(h_1,h_3),(h_2,h_3)\in X$.
By inductively iterating the process we obtain the desired infinite set $H:=\{h_n\mid n\in\N\}$.
\end{proof}

Clearly, the above property correspond to \emph{Ramsey's Theorem} for pairs.

\begin{corollary}[Ramsey's Theorem for pairs]
Let $[\N]^2=C_1\cup\ldots\cup C_r$ be a finite coloring of the pairs of natural numbers.
Then there exists an infinite set $H$ such that the pairs $[H]^2\subseteq C_i$ 
are monochromatic.\footnote
{Here we are identifying the set of (unordered) pairs $[\N]^2:=\{\{n,m\}\subseteq \N\mid n\ne m\}$
with the set of (ordered) pairs $\{(n,m)\in \N\times\N\mid n<m\}$.}
\end{corollary}

\begin{proof}
Pick any infinite $\alpha\in\hN\setminus\N$, and let $C_i$ be the color
such that $(\alpha,{}^*\alpha)\in{}^{**}C_i$. Then apply the previous theorem.
\end{proof}

In a similar way, one can also prove the general case of Ramsey's Theorem,
as a straight consequence of the following property.

We write ${}^{n*}A$ to denote
the $n$-th iterated hyper-extension of an object $A$.

\begin{itemize}
\item
\emph{Let $k\in\N$, let $X\subseteq\N^k$, and let $\alpha\in\hN\setminus\N$ be infinite.
If $(\alpha,{}^*\alpha,\ldots,{}^{(k-1)*}\alpha)\in{}^{k*}X$ 
then there exists an infinite set $H$ such that}
$$[H]^k=\{(h_1,\ldots,h_k)\in H^k\mid h_1<\ldots<h_k\}\subseteq X.$$
\end{itemize}

The $u$-equivalence is extended in a natural way to numbers that
belongs to iterated hyper-extensions of the natural numbers.
For example, if $\alpha\in{}^{**}\N$ and $\beta\in{}^{****}\N$, then
$\alpha\,\ueq\,\beta$ means that $\alpha\in{}^{**}A\Leftrightarrow\beta\in{}^{****}A$ for every $A\subseteq\N$.

\smallskip
The different levels of infinity are compatible with $u$-equivalence.

\begin{proposition}
Let $\alpha\in\hN$. Then $\alpha\,\ueq\,{}^{k*}\alpha$ for every $k\in\N$.
\end{proposition}

\begin{proof}
For every $A\subseteq\N$, by repeated applications of \emph{transfer} one 
obtains the following equivalences:
$$\alpha\in\hA\Leftrightarrow{}^*\alpha\in{}^{**}A\Leftrightarrow\ \ldots\ 
\Leftrightarrow{}^{*k}\alpha\in{}^{(k+1)*}A.$$
\end{proof}

Later, we will use the following facts.

\begin{theorem}
Let $\alpha_i,\beta_i\in\hN\setminus\N$ be such that $\alpha_i\ueq\beta_i$ for every $i\in\N_0$.
Then also the following sums are $\ueq$-equivalent:
\begin{itemize}
\item
$\alpha_0+{}^*\alpha_1\ \ueq\ \beta_0+{}^*\beta_1$.
\item
$\alpha_0+{}^*\alpha_1+{}^{**}\alpha_2\ \ueq\ \beta_0+{}^*\beta_1+{}^{**}\beta_2$.
\item
$\alpha_0+{}^*\alpha_1+{}^{**}\alpha_2+{}^{***}\alpha_3\ \ueq\ 
\beta_0+{}^*\beta_1+{}^{**}\beta_2+{}^{***}\beta_3$, and so forth.
\end{itemize}
\end{theorem}

\begin{proof}
We will only prove the first property; the other ones are entirely similar.
Fix a set $A\subseteq\N$. If $\alpha_0+{}^*\alpha_1\in{}^{**}A$
then $\alpha_0\in{}^*\Lambda$ where
$\Lambda:=\{x\in\N\mid x+{}^*\alpha_1\in{}^{**}A\}$, and hence
$\beta_0\in{}^*\Lambda$. Now observe that 
\begin{multline*}
\Lambda=\{x\in\N\mid x+\alpha_1\in\hA\}=\{x\in\N\mid \alpha_1\in{}^*(A-x)\}=
\\
=\{x\in\N\mid \beta_1\in{}^*(A-x)\}=\{x\in\N\mid \beta_1+x\in\hA\}.
\end{multline*}
So, $\beta_0\in{}^*\Lambda$ means that $\beta_0+{}^*\beta_1\in{}^{**}A$.
\end{proof}

\noindent
\emph{Caution}: If we do not use different levels of infinity, $u$-equivalence 
is not preserved. In fact, in general, $\alpha\ueq\beta$ and $\alpha'\ueq\beta'$
do not imply that $\alpha+\alpha'\,\ueq\,\beta+\beta'$.

\smallskip
The previous theorem also holds for multiplication in place of addition, 
and more generally, for arbitrary functions.
That is, if $\alpha_i\ueq\beta_i$ and $f:\N^k\to\N$, then 
$$f(\alpha_0,{}^*\alpha_1,\ldots,{}^{(k-1)*}\alpha_{k-1})\,\ueq\,
f(\beta_0,{}^*\beta_1,\ldots,{}^{(k-1)*}\beta_{k-1}).$$

Let us see an example of application of different levels of infinity.
We will show the existence of monochromatic exponential triples,
by using the multiplicative version of \emph{Brauer's Theorem} as a starting point.\footnote
{Recall that Brauer's Theorem is the extension of van der Waerden's Theorem
where also the common distance is assumed to be of the same color
of the arithmetic progression.}

\begin{theorem}\label{exponential}
In every finite coloring of  $\N$ there exists a monochromatic
exponential triple $x, y, y^x$.\footnote
{~This result is proved in \cite{dra0} by using ultrafilters. However, it
was originally proved by nonstandard methods along the lines presented here.}
\end{theorem}

\begin{proof}
By multiplicative Brauer's Theorem there exist
$\alpha,\delta\in\hN$ such that $\alpha\ueq\delta\ueq\alpha\cdot\delta^\ell$ for every $\ell\in\N$.
By \emph{transfer} it follows that 
${}^*\alpha\ueq{}^*\delta\ueq{}^*\alpha\cdot{}^*\delta^\nu$ for every $\nu\in\hN$,
and in particular ${}^*\alpha\ueq{}^*\alpha\cdot{}^*\delta^\alpha$.
Now let $\xi:=({}^*\delta)^\alpha$ and $\eta:={}^*\xi=({}^{**}\delta)^{{}^*\alpha}$. 
Then $\xi\ueq\eta$, and besides,
$$\eta^\xi=({}^{**}\delta)^{{}^*\alpha\cdot{}^*\delta^\alpha}\ueq({}^{**}\delta)^{{}^*\alpha}=\eta.$$
Since $\xi\ueq\eta\ueq\eta^\xi$, we obtained the desired result.
\end{proof}

The arguments used in the previous proof, as always happens when using iterated 
hyper-extensions and $u$-equivalence, can be reformulated in the language of ultrafilters.
In this case, one still obtains a fairly simple proof, that can be found in
\S \ref{sec-ultrafilters} (see Remark \ref{ultraexponential});
however, in other cases the translation can become much more complex and involved.

\section{Hypernatural numbers and ultrafilters}\label{sec-ultrafilters}

\emph{Note to the reader}: 
Those who are unfamiliar with ultrafilters or simply dislike them can safely skip this section, 
as well as the remarks on ultrafilters found in subsequent sections.
Indeed, the sole aim here is to highlight the connections between the use of
hypernatural numbers and the use of ultrafilters in the study of problems in 
arithmetic Ramsey theory.

\smallskip
As already suggested by the material presented so far, the connection between 
the hypernatural numbers $\hN$ of nonstandard analysis
and Ramsey problems is very strong. This is based on the observation
that -- in a precise sense -- every hypernatural number corresponds 
to an ultrafilter on $\N$; furthermore, the sums of ultrafilters have a precise counterpart
as sums of numbers in iterated hyperextensions of $\N$.
Given that the algebra on the space of ultrafilters $\beta\N$ has produced
an enormous amount of results in this area (see the volume \cite{hs}),
it is not surprising that useful applications can also be found using hypernatural numbers.

A significant advantage of the nonstandard approch, compared to working 
with ultrafilters, is that in many respects they behave 
like natural numbers; in fact, $\hN$ is the positive part of a discretely ordered 
commutative ring and satisfies exactly the same ``elementary" (\emph{i.e.}, first-order) properties 
as the set of natural numbers. 
This leads to a simpler formalism that, in practice, allows for much easier handling with respect to 
than the algebraic manipulations of ultrafilters.

Another advantage of working with hypernatural numbers is that
it often facilitates heuristic processing;
in fact, the nonstandard framework offers a different perspective on problems, 
and this has helped guide solutions to problems on several occasions.
We will provide several examples later to illustrate this point.

We remark that the connection between hypernatural numbers $\hN$ and ultrafilters
is just a special case of a more general picture.
Indeed, the entire field of nonstandard analysis can be seen
as a general and uniform method that incorporates ultrafilter techniques.
As demonstrated by a vast literature, applications of ultrafilters
and ultrapowers are found throughout mathematics, including model theory, 
set theory, topology, algebra, functional analysis, as well as
various aspects of combinatorics and Ramsey theory; 
all of these applications find a natural place in the context of 
nonstandard analysis (see the book \cite{lw} for an overview).

Let us start by recalling the fundamental notions.

\begin{definition}
An \emph{ultrafilter} $\U$ on $\N$ is a family of subsets such that:
\begin{itemize}
\item
$\emptyset\notin\U$ and $\N\in\U$.
\item
If $A,B\in\U$ then also $A\cap B\in\U$.
\item
If $A\in\U$ and $B\supseteq A$ then also $B\in\U$.
\item
For every $A\subseteq\N$, either $A\in\U$ or the complement $A^c\in\U$.
\end{itemize}
\end{definition}

Here we focus on ultrafilters on $\N$, but clearly the same properties as above
are used to define ultrafilters on arbitrary sets $I$.

Families that satisfy the first three properties above are called \emph{filters}.
It is shown that ultrafilters are precisely the maximal filters with respect to inclusion.

A fundamental example is the \emph{Fr\'echet filter} 
$\mathcal{F}=\{A\subseteq\N\mid A^c\ \text{is finite}\}$ of the cofinite sets.
Trivial examples of ultrafilters are given by the \emph{principal ultrafilters}
$\UU_n=\{A\subseteq\N\mid n\in A\}$ that are ``generated" by elements $n\in\N$. 
It is easily shown that an ultrafilter $\U$ is non-principal if and only if it only contains infinite sets
if and only if it extends the Fr\'echet filter.
 
The existence of non-principal ultrafilters
is proved by a simple application of \emph{Zorn's Lemma}; in fact,
one observes that non-principal ultrafilters are precisely the maximal elements
in the poset given by family of filters extending the Fr\'echet filter, 
and where the partial order is the inclusion.

The initial observation that builts a bridge to the nonstandard context
is the fact that every hypernatural number uniquely determines an ultrafilter on $\N$, 
similarly to how natural numbers generate principal ultrafilters.

\begin{definition}
The \emph{ultrafilter generated} by a point $\nu\in\hN$ is the family of sets:
$$\UU_\nu:=\{A\subseteq\N\mid\nu\in\hA\}.$$
\end{definition}

Since hyper-extensions are coherent with unions, intersections, and complements,
it is easy to verify that $\UU_\nu$ is actually an ultrafilter.
Besides, it is also easily checked that $\UU_\nu$ is non-principal if and 
only if $\nu\in\hN\setminus\N$ is infinite. 

By \emph{enlargement}, the points of $\hN$ generate all possible ultrafilters on $\N$.

\begin{proposition}
Assume $\mathfrak{c}$-enlargement. Then for every ultrafilter $\U$ on $\N$
there exists $\nu\in\hN$ such that $\UU_\nu=\U$.
\end{proposition}

\begin{proof}
Note that every ultrafilter $\U$ is a family of subsets of $\N$ with the finite intersection property.
Since $|\U|\le|\PP(\N)|=\mathfrak{c}$, by $\mathfrak{c}$-\emph{enlargement}
we can pick $\nu\in\bigcap_{A\in\U}\hA$. It is then readily verified that $\UU_\nu=\U$.
\end{proof}

The following property follows easily from the definitions,
and justifies the use of the name ``$u$-equivalence", which stands for ``ultrafilter-equivalence".

\begin{itemize}
\item
\emph{Let $\nu,\mu\in\hN$. Then $\nu\ueq\mu$ if and only if the generated ultrafilters $\UU_\nu=\UU_\mu$.}
\end{itemize}

\begin{remark}
Recall that the \emph{ultrafilter shift} of a set $A\subseteq\N$ by an
ultrafilter $\U$ on $\N$ is defined as
$$A-\U:=\{n\in\N\mid A-n\in\U\}.$$
Note that this is just a remote realm (see Definition \ref{def-remoterealm}).
In fact, for every $\nu\in\hN$, the remote realm $A_\nu=A-\UU_\nu$ 
is the ultrafilter shift of $A$ by the ultrafilter generated by $\nu$.
To see this, observe that $n\in A-\UU_\nu$ if and only if
$A-n\in\UU_\nu$ if and only if $\nu\in{}^*(A-n)$ if and only if
$\nu+n\in\hA$ if and only if $n\in A_\nu$.\footnote
{The notion of ultrafilter shift was introduced by M. Beiglb\"{o}ck in \cite{bei}.}
\end{remark}

\begin{remark}\label{PR-ultrafilters}
The nonstandard characterization of partition regular families given in
Theorem \ref{PRcharacterizations}, corresponds to the following well-know result
about ultrafilters (see for example \cite[Theorem 3.11]{hs}).
\begin{itemize}
\item
\emph{A family $\mathfrak{P}\subseteq\mathcal{P}(\N)$ is partition regular
if and only if there exists an ultrafilter $\mathcal{U}$ on $\N$ that
is a ``PR-witness" for $\mathfrak{P}$, \emph{i.e.}, 
such that for every $A\in\U$ there exists $P\in\mathfrak{P}$ with $P\subseteq A$.}
\end{itemize}
To see this, observe that a number $\nu\in\hN$ has the property
that $P\subseteq A$ for some $P\in\mathfrak{P}$ whenever $\nu\in\hA$,
if and only if the generated ultrafilter $\UU_\nu$ is a PR-witness for $\mathfrak{P}$.
\end{remark}

Recall the following notion.

\begin{definition}
Let $f:\N\to\N$ be a function and let $\U$ be an ultrafilter on $\N$.
The \emph{image ultrafilter} $f(\U)$ is the ultrafilter on $\N$
defined by setting $A\in f(\U)\Leftrightarrow f^{-1}(A)\in\U$.
\end{definition}

More generally, for any function $f:I\to J$ and for any ultrafilter $\U$ on $I$,
one similarly defines the image ultrafilter $f(\U):=\{f^{-1}(A)\mid A\in\U\}$ on $J$.

\begin{remark}
We have seen the following property of $u$-equivalence (it is item (5) in Theorem \ref{u-equivalenceproperties}):
\begin{itemize}
\item
\emph{If $\hf(\nu)\ueq\nu$ then $\hf(\nu)=\nu$ for every $f:\N\to\N$.}
\end{itemize}
This corresponds to the following (nontrivial) fact about ultrafilters:
\begin{itemize}
\item
\emph{If the image ultrafilter $f(\U)=\U$ then $f$ is $\U$-almost everywhere the identity,
\emph{i.e.}, $\{n\in\N\mid f(n)=n\}\in\U$.}
\end{itemize}
To see this, note that $f(\UU_\nu)=\UU_{\hf(\nu)}$ and
that $\{n\in\N\mid f(n)=n\}\in\UU_\nu$ if and only if $\hf(\nu)=\nu$.
\end{remark}

\subsection{Hypernatural numbers and algebra in $\beta\N$}

Recall the following fundamental operation between ultrafilters.

\begin{definition}
Let $\U,\V$ be ultrafilters on $\N$. Their \emph{tensor product} $\U\otimes\V$
is the ultrafilter on $\N\times\N$ defined by setting, for every $X\subseteq\N\times\N$:
$$X\in\U\otimes\V\ \Longleftrightarrow\ 
\{n\in\N\mid X_n\in\V\}\in\U$$
where $X_n:=\{m\in\N\mid (n,m)\in X\}$ is the $n$-fiber of $X$.
\end{definition}

The operation $\otimes$ is associative (assuming the usual
identification of the Cartesian products $\N\times(\N\times\N)$ and $(\N\times\N)\times\N$).
However, we remark that $\otimes$ is not commutative.

Tensor products correspond to the use of different levels of infinity in a nonstandard setting. Precisely:

\begin{proposition}
Let $\nu,\mu\in\hN$. Then the tensor product
$\UU_\nu\otimes\UU_\mu=\UU_{(\nu,{}^*\mu)}$ is the
ultrafilter generated by the pair $(\nu,{}^*\mu)$.
\end{proposition}

\begin{proof}
For each $X\subseteq\N\times\N$, one has the following chain of equivalences:
\begin{multline*}
X\in\UU_\nu\otimes\UU_\mu\Leftrightarrow
\{n\in\N\mid X_n\in\UU_\mu\}\in\UU_\nu\Leftrightarrow
\{n\in\N\mid \mu\in{}^*(X_n)\}\in\UU_\nu\Leftrightarrow
\\
\Leftrightarrow\{n\in\N\mid (n,\mu)\in\hX\}\in\UU_\nu\Leftrightarrow
\nu\in{}^*\{n\in\N\mid (n,\mu)\in\hX\}\Leftrightarrow
(\nu,{}^*\mu)\in{}^{**}X.
\end{multline*}
\end{proof}

The following sum of ultrafilters on $\N$ revealed a fundamental tool for applications
in Ramsey theory.

\begin{definition}
Let $\U,\V$ be ultrafilters on $\N$. Their \emph{sum} $\U\oplus\V$
is the ultrafilter on $\N$ defined by setting, for every $A\subseteq\N$:
$$A\in\U\oplus\V\ \Longleftrightarrow\ 
\{n\in\N\mid A-n\in\V\}\in\U$$
where $A-n:=\{m\in\N\mid n+m\in A\}$ is the leftward shift of $A$ by $n$.
\end{definition}

It can be directly verified that $\U\oplus\V=\text{Sum}(\U\otimes\V)$ is
the image ultrafilter of the tensor product under the sum function $\text{Sum}:\N\times\N\to\N$
where $(n,m)\mapsto n+m$. Consequently, the sums are associative, \emph{i.e.},
$\U\oplus(\V\oplus\W)=(\U\oplus\V)\oplus\W$ for all ultrafilters $\U,\V,\W$ on $\N$;
however, in general $\U\oplus\V\ne\V\oplus\U$.

Furthermore, sums of ultrafilters correspond to sums of hypernatural numbers of different levels of infinity,
thus allowing to embed algebra of ultrafilters in our nonstandard setting.

\begin{proposition}
Let $\nu,\mu\in\hN$. Then the sum of ultrafilters
$\UU_\nu\oplus\UU_\mu=\UU_{\nu+{}^*\mu}$ is the
ultrafilter generated by the number $\nu+{}^*\mu\in{}^{**}\N$.
\end{proposition}

\begin{proof}
$\UU_\nu\oplus\UU_\mu=\text{Sum}(\UU_\nu\otimes\UU_\mu)=
\textrm{Sum}(\UU_{(\nu,{}^*\mu)})=\UU_{{}^*{Sum}(\nu,{}^*\mu)}=\UU_{\nu+{}^*\mu}$.
\end{proof}

The previous proposition builds a bridge between nonstandard proofs
that use different levels of infinity, and ``standard" proofs that involve sums of ultrafilters.

\begin{remark}\label{ultraexponential}
In the previous section we proved the following result:

\smallskip
\noindent
\textbf{Theorem \ref{exponential}}.
\emph{In every finite coloring of  $\N$ there exists a monochromatic
exponential triple $x, y, y^x$.}

\smallskip
The (short) nonstandard proof that we provided can be translated into the language of ultrafilters
as follows.
\begin{proof}[Ultrafilter proof]
For every $A\subseteq\N$ and for every $\ell\in\N$, consider the following set of ordered pairs:
$$\Lambda(A,\ell):=\{(a,d)\in\N^2\mid \text{either}\ a, d, a\cdot d, \ldots, a\cdot d^\ell\in A\ \text{or}\ 
a, d, a\cdot d, \ldots, a\cdot d^\ell\in A^c\}.$$

We observe that the family $\mathcal{F}:=\{\Lambda(A,\ell)\mid A\subseteq\N,\ \ell\in\N\}$ has the
finite intersection property. Indeed, given $A_1,\ldots,A_n\subseteq\N$ and $\ell_1,\ldots,\ell_n\in\N$,
let $\N=C_1\cup\ldots\cup C_r$ be the finite partition generated by $A_1,\ldots,A_n$,
and let $\ell=\max\{\ell_i\mid i=1,\ldots,n\}$.
By multiplicative Brauer's Theorem there exists a monochromatic pattern
$$b, c, b\cdot c, \ldots, b\cdot c^\ell\subseteq C_i.$$
Then clearly $(b,c)\in\bigcap_{j=1}^n\Lambda(A_j,\ell_j)$.
Now pick an ultrafilter $\W$ on $\N^2$ that extends the family $\F$.
Since the sets $\Lambda(A,\ell)\subseteq (A\times A)\cup (A^c\times A^c)$, it follows 
that the canonical projections $\pi_1(\W)=\pi_2(\W)=\U$ are equal.
Note also that $\U$ is a ``witness" of multiplicative Brauer's Theorem, \emph{i.e.}:
\begin{itemize}
\item
For every $B\in\U$ and for every $\ell\in\N$ there exists
a pattern $a, d, a\cdot d, \ldots, a\cdot d^\ell\in\B$.
\end{itemize}

Let $\V:=\text{Exp}(\U\otimes\U)$ be the image ultrafilter of
the tensor power $\U\otimes\U$ under the map $\text{Exp}:(x,y)\mapsto y^x$.
We claim that $\V$ is a ``witness" of the partition regularity of exponential triples, \emph{i.e.},
for every $C\in\V$ there exists a pattern $x, y, y^x\in C$.
This will yield the desired partition regularity result.

By definition, $C\in\V$ if and only if $\text{Exp}^{-1}(C)\in\U\otimes\U$, and this means that
$\widehat{C}:=\{n\in\N\mid C_n\in\U\}\in\U$, where $C(n):=\{m\in\N\mid m^n\in C\}$.
Pick $a\in \widehat{C}$. Since $C(a)\cap\widehat{C}\in\U$
there exist elements $b, c, b\cdot c^a\in C(a)\cap\widehat{C}$.
Then pick an element $d\in C(b)\cap C(b\cdot c^a)\in\U$, and let
$x:=c^a$ and $y:=d^b$. Note that $x\in C$ since $c\in C(a)$; similarly,
$y\in C$ since $d\in C(b)$. Finally, $y^x=d^{b\cdot c^a}\in C$,
since $d\in C(b\cdot c^a)$.
\end{proof}
\end{remark}

Similarly to sums $\U\oplus\V$, it is possible to define products
$\U\odot\V$ between ultrafilters on $\N$. More generally,
it is also possible to define an associative operation $\ostar$ in the space $\beta S$ of ultrafilters on
any semigroup $(S,\star)$, by letting
for all $\U,\V\in\beta S$ and for every $A\subseteq S$:
$$A\in\U\ostar\V\ \Longleftrightarrow\ \{s\in S\mid \{t\in S\mid s\star t\in A\}\in\V\}\in\U.$$
The resulting algebras on the space of ultrafilters have been widely studied, 
finding numerous applications in Ramsey theory (see \cite{hs} and the references therein).

An important property of algebra in the space of ultrafilters
is the existence of idempotent elements, which have proven crucial in applications.

\begin{theorem}[Ellis' Lemma]
Let $(S,\star)$ be a semigroup. Then there exist idempotent elements $\U\ostar\U=\U$
in the corresponding semigroup $(\beta S,\ostar)$ on the space of ultrafilters.
\end{theorem}

\subsection{Hypernatural numbers and topological dynamics}

We conclude this section with a remark on the connection between 
hypernatural numbers and topological dynamics.
Although this topic is beyond the scope of this article, we feel it is worth mentioning.

\smallskip
There is a natural dynamics on the space $\beta\N$ that has been studied,
namely the one obtained by considering the flow $\U\mapsto 1\oplus\U$.
Such a topological dynamics have been instrumental in several
relevant applications in Ramsey theory.\footnote
{See, \emph{e.g.}, \cite[Chapter 19]{hs} and references therein.
See also \cite{bl} for several nice examples of dynamical proofs
of Ramsey theory results performed in the space $\beta\N$.}
Similarly, one can consider a natural dynamics on the compact regular space
$(\hN,\tau_S)$ of the hypernatural numbers (see Remark \ref{hNtopologicalspace}), 
by considering the shift $\s:\nu\mapsto\nu+1$.
Also in this case, the nonstandard setting gives an alternative
viewpoint that has several interesting aspects.
Probably, the most relevant one is the fact that return time sets
$N(\nu,U):=\{n\in\N\mid \nu+n\in U\}$ and neighborhoods $U$ 
are both sets of (hyper)natural numbers, and this determines
an interplay between these two notions that is highlithed when partition regularity problems
on $\N$ are considered. 

For example, an idempotent point $\nu\in\hN$,
\emph{i.e.}, a point that generates an idempotent ultrafilter $\UU_\nu\oplus\UU_\nu=\UU_\nu$,
can be seen as a sort of ``self-recurrent" point, \emph{i.e.}, a recurrent point
of the dynamical system $(\hN,\s)$ with the special property
that the return time set into any of its neighborhoods 
must contain a point that belongs itself to the neighborhood. In other words,
$\nu\in\hN$ is idempotent if and only if for every neighborhood $U$ of $\nu$,
one has $N(\nu,U)\cap U\ne\emptyset$.
A nonstandard treatment of topological dynamics on the hypernatural numbers
is outside the scope of this paper; however,
the interested reader is refereed to \cite{isaac}, focused on a nonstandard take of central sets,
and to \cite{pb,dp} for more comprehensive treatments,
where also the nonstandard version of the Bernoulli dynamics is studied.

\section{Idempotent numbers}\label{sec-idempotents}

We now introduce a special class of hypernatural numbers that are instrumental
in applications.

\begin{definition}
Let $\gamma\in\hN$. We say that $\gamma$ is \emph{additively idempotent}
(or simply \emph{idempotent}) if $\gamma+{}^*\gamma\,\ueq\,\gamma$.
Similarly, we say that $\gamma$ is \emph{multiplicatively idempotent}
if $\gamma\cdot{}^*\gamma\,\ueq\,\gamma$.
\end{definition}

\begin{remark}
Note that $\gamma\in\hN$ is idempotent if and only if the 
generated ultrafilter $\UU_\gamma=\UU_\gamma\oplus\UU_\gamma$ is idempotent.
Indeed, recall that $\UU_\gamma\oplus\UU_\gamma=\UU_{\gamma+{}^*\gamma}$.
Thus the existence of idempotent numbers is a consequence of Ellis' Lemma.
\end{remark}

\begin{proposition}
If $\gamma\in\hN$ is additively idempotent, then $\gamma$ is a multiple of every finite $n\in\N$.
\end{proposition}

\begin{proof}
Recall that $\gamma\ueq{}^*\gamma$; so, if $\gamma\equiv i\!\mod n$ then
also ${}^*\gamma\equiv i\!\mod n$. Then, since $\gamma+{}^*\gamma\ueq\gamma$,
we have $2i\equiv i\!\mod n$, and hence $i\equiv 0\!\mod n$.
\end{proof}

Note that if $\gamma$ is idempotent, then also 
$\gamma+{}^*\gamma+{}^{**}\gamma\,\ueq\,\gamma$,
and $\gamma+{}^*\gamma+{}^{**}\gamma+{}^{***}\gamma\,\ueq\,\gamma$,
and so forth.

Historically, idempotent ultrafilters were isolated to provide a short proof 
of Hindman's Theorem \cite{hi0}, a cornerstone of arithmetic Ramsey theory.
Below we provide a ``nonstandard" proof.

\begin{theorem}[Hindman]\label{hindman}
Let $\N=C_1\cup\ldots\cup C_r$ be a finite coloring.
Then there exists an increasing sequence $(a_i\mid i\in\N)$ such that the
corresponding set of finite sums is monochromatic, \emph{i.e.}, there exists a color $C_j$ such that
$$\text{FS}(a_i)_{i=1}^\infty:=\{a_{i_1}+\ldots+a_{i_k}\mid i_1<\ldots<i_k\}\subseteq C_j.$$
\end{theorem}

\begin{proof}
Pick an idempotent $\gamma\in\hN$, and let $C_j$ be the color such that $\gamma\in{}^*C_j$.
We recursively define an increasing sequence $(a_n\mid n\in\N)$ such that for every $n$:
\begin{itemize}
\item
$\text{FS}(a_i)_{i=1}^n:=\{a_{i_1}+\ldots+a_{i_k}\mid i_1<\ldots<i_k\le n\}\subseteq C_j$.
\item
$\text{FS}(a_i)_{i=1}^n+\gamma\subseteq{}^*C_j$.
\end{itemize}

Since $\gamma\,\ueq\,\gamma+{}^*\gamma$, 
from $\gamma\in{}^*C_j$ it follows that $\gamma+{}^*\gamma\in{}^{**}C_j$.
By \emph{backward transfer} we obtain the existence of an element $a_1\in C_j$
such that $a_1\in C_j$ and $a_1+\gamma\in{}^*C_j$.
This completes the base case $n=1$, since trivially $\text{FS}(a_i)_{i=1}^1=\{a_1\}$.

At the inductive step, by the hypothesis we know that the infinite $\gamma\in\hN$ satisfies
the following property: ``For every $x\in\text{FS}(a_i)_{i=1}^n$ it is
$x+\gamma\in{}^*C_j$, and hence also $x+\gamma+{}^*\gamma\in{}^{**}C_j$."
Note that $\gamma$ is infinite and hence $\gamma>a_n$. So, by \emph{backward transfer},
we obtain the existence of a natural number $a_{n+1}>a_n$ such that
$x+a_{n+1}\in C_j$ and $x+a_{n+1}+\gamma\in {}^*C_j$ for every $x\in\text{FS}(a_i)_{i=1}^n$.
This means that $\text{FS}(a_i)_{i=1}^n+a_{n+1}\subseteq C_j$ and $\text{FS}(a_i)_{i=1}^n+a_{n+1}+\gamma\in{}^*C_j$.
Finally, note that $\text{FS}(a_i)_{i=1}^n+a_{n+1}=\text{FS}(a_i)_{i=1}^{n+1}$.
\end{proof}

\smallskip
A recent result whose proof was obtained by using idempotent numbers is the following:\footnote
{See \cite{dlmrv}, where an infinitary version of this result is proved.}

\begin{theorem}\label{sumsquotients}
In every finite coloring of  $\N$ there exists a monochromatic
pattern of the form $x, y, x+y, \frac{y}{x}$.
\end{theorem}

\begin{proof}
Pick an idempotent number $\gamma\in\hN$, and let
$\xi:=\frac{{}^*\gamma}{\gamma}$ and $\eta:=\frac{{}^{**}\gamma}{\gamma}$.
Since $\gamma$ is a multiple of every $n\in\N$,
by \emph{transfer}, ${}^*\gamma$ is a multiple of every $\nu\in\hN$
and ${}^{**}\gamma$ is a multiple of every $\nu\in{}^{**}\N$. In particular,
$\xi\in{}^{**}\N$ and $\eta\in{}^{***}\N$.
Since $\gamma\ueq\gamma+{}^*\gamma$, we have that
$\xi+\eta=\frac{{}^*\gamma+{}^{**}\gamma}{\gamma}\ueq\frac{{}^*\gamma}{\gamma}=\xi$
and $\frac{\eta}{\xi}=\frac{{}^{**}\gamma}{{}^*\gamma}\ueq\frac{{}^*\gamma}{\gamma}=\xi$.
This completes the proof that $\xi\ueq\eta\ueq\xi+\eta\ueq\frac{\eta}{\xi}$.
\end{proof}

Below is the translation of the above proof into the language of ultrafilters.

\begin{proof}[Ultrafilter proof of Theorem \ref{sumsquotients}]
Pick an idempotent ultrafilter $\U$ in $(\beta\N,\oplus)$.
It is well-known that the set of multiples $n\N:=\{nk\mid k\in\N\}\in\U$ for every $k\in\N$.
As a consequence, the set $\{n\in\N\mid\{m\in\N\mid n\mathrel{|}m\}\in\U\}\in\U$, and hence
$$D:=\{(n,m)\in\N^2\mid n\mathrel{|}m\}\in\U\otimes\U.$$
Now let $\text{Div}:\N^2\to\N$ be the function where $(n,m)\mapsto\frac{m}{n}$ if $(n,m)\in D$
and $(n,m)\mapsto 1$ otherwise, and let
$\V:=\text{Div}(\U\otimes\U)$ be the image ultrafilter of the tensor power $\U\otimes\U$.
Given a coloring $\N=C_1\cup\ldots\cup C_r$, pick the color $C=C_i$ such that $C\in\V$. Then 
$$X:=\left\{(n,m)\in\N^2\ \middle |\  \frac{m}{n}\in C\right\}=D\cap\text{Div}^{-1}(C)\in\U\otimes\U.$$ 

Recall that, by definition, a set $Y\in\U\otimes\U$ if and only if
$\widehat{Y}:=\{n\in\N\mid Y_n\in\U\}\in\U$ where
$Y_n=\{m\in\N\mid (n,m)\in Y\}$.
Recall also that, since $\U$ is idempotent,
a set $B\in\U$ if and only if $B':=\{n\in\N\mid B-n\in\U\}\in\U$.

Since $\widehat{X}:=\{n\in\N\mid X_n\in\U\}\in\U$, we can pick an element $a\in\widehat{X}$.
Then $X_a\in\U$, and hence $(X_a)'\in\U$.
Pick $b\in X_a\cap(X_a)'\cap\widehat{X}\in\U$,
so that $\frac{b}{a}\in C$, $X_a-b\in\U$ and $X_b\in\U$.
Then pick $c\in X_a\cap X_b\cap (X_a-b)\in\U$, so that
$\frac{c}{a},\frac{c}{b},\frac{b+c}{a}\in C$.
If we let $x:=\frac{b}{a}$ and $y:=\frac{c}{a}$, then $x, y, x+y, \frac{y}{x}\in C$, as desired.
\end{proof}

As a curiosity, we obtain the following

\begin{corollary}
In every finite coloring of  $\N$ there exists a monochromatic
pattern of the form $2^a, b, 2^a b, \sqrt[a]{b}$.
\end{corollary}

\begin{proof}
By the previous theorem there exist hypernatural numbers $\xi\ueq\eta\ueq\xi+\eta\ueq\frac{\eta}{\xi}$.
Then if we let $\alpha:=\xi$ and $\beta:=2^\eta$, also the following four hypernatural numbers
are $u$-equivalent to each other:
$$2^\alpha=2^\xi,\quad \beta=2^\eta,\quad 2^\alpha\beta=2^{\xi+\eta},\quad 
\sqrt[\alpha]{\beta}=2^{\frac{\eta}{\xi}}.$$
\end{proof}

Combining monochromatic sums and products is currently one of the hottest areas of
research in arithmetic Ramsey theory. In fact, even the following ``simple"
problem is open to this day:

\begin{itemize}
\item
\emph{Is it true that in every finite coloring of the natural numbers one
finds a monochromatic pattern of the form $x, y, x+y, x\cdot y$?}
\end{itemize}

\subsection{Monochromatic products of sums}

As another example of application of the method of iterated hyper-extensions,
we present here a nonstandard proof of the existence of
certain monochromatic patterns made of products of sums; the proof
uses a special kind of multiplicatively idempotent elements.

Recall that if $A=\{a_1<\ldots<a_k\}$ is a finite set of natural numbers,
the corresponding set of finite sums is defined by setting:
$$\text{FS}(A)=\text{FS}(a_i)_{i=1}^k:=
\left\{\sum_{a\in F}a\ \middle\vert\ \emptyset\ne F\subseteq A\right\}=
\{a_{i_1}+\ldots+a_{i_s}\mid i_1<\ldots<i_s\le k\}.$$

If $A=\{a_1<a_2<\ldots<a_n<a_{n+1}<\ldots\}$ is infinite,
one defines:
\begin{multline*}
\text{FS}(A)=\text{FS}(a_n)_{n=1}^\infty:=\bigcup_{k\in\N}\text{FS}(a_i)_{i=1}^k=
\\
=
\left\{\sum_{a\in F}a\ \middle\vert\ \emptyset\ne F\subseteq A\ \text{finite}\right\}=
\{a_{n_1}+\ldots+a_{n_k}\mid n_1<\ldots<n_k\}.
\end{multline*}

A natural generalization is obtained by considering a sequence 
$(A_n\mid n\in\N)$ of nonempty sets of natural numbers; in this case we set
$$\text{FS}(A_n)_{n=1}^\infty:=\{a_{n_1}+\ldots+a_{n_k}\mid a_{n_j}\in A_{n_j}, n_1<\ldots<n_k\},$$
The same notion is also defined for products:
$$\text{FP}(A_n)_{n=1}^\infty:=\{a_{n_1}\cdot\ldots\cdot a_{n_k}\mid a_{n_j}\in A_{n_j}, n_1<\ldots<n_k\},$$

\begin{theorem}\label{productsums}
For every finite coloring $\N=C_1\cup\ldots\cup C_r$ and for every $f:\N\to\N$
there exists a sequence of
nonempty sets $(A_n\mid n\in\N)$ where $|A_n|=f(n)$ and such that
the following sets of products of sums is monochromatic:\footnote
{This property is essentially a reformulation of Theorem 3.1 in \cite{bo}. See Remark \ref{equivalentBowen}.}
$$\text{FP}(\text{FS}(A_n))_{n=1}^\infty=\left\{
\prod_{i=1}^{k}
\left(
\sum_{a\in F_i} a
\right)
\;\middle|\;
\emptyset \ne F_i \subseteq A_{n_i},
\
n_1 < \cdots < n_k
\right\}\subseteq C_j.
$$ 
We can assume that $\max\text{FP}(\text{FS}(A_1),\ldots,\text{FS}(A_n))<\min A_{n+1}$ for every $n$.
\end{theorem}

We observe that already in the particular case when $f(1)=f(2)=f(3)=2$,
and only restricting to consider $A_1=\{a_1<a_2\}$, $A_2=\{b_1<b_2\}$
and $A_3=\{c_1<c_2\}$, one obtains
a rich monochromatic configuration of sums and products. Precisely:

\begin{corollary}\label{cor-sumproduct}
In every finite coloring of $\N$ there exist a monochromatic pattern of the 
following form, where $a_1<a_2<a_1+a_2<b_1<b_2<b_1+b_2<c_1<c_2$: 
\begin{multline*}
a_1,\ a_2,\ b_1,\ b_2,\ c_1,\ c_2,\ 
a_1+a_2,\ b_1+b_2,\ c_1+c_2,
\\
a_1 b_1,\ a_1 b_2,\ a_1 c_1,\ a_1 c_2,\ a_2 b_1,\ a_2 b_2,\ a_2 c_1,\ a_2 c_2,\
b_1 c_1,\ b_1 c_2,\ b_2 c_1,\ b_2 c_2,
\\
(a_1+a_2) b_1,\ (a_1+a_2) b_2,\ (a_1+a_2) c_1,\ (a_1+a_2)c_2,
\\
a_1(b_1+b_2),\ a_2(b_1+b_2),\ (b_1+b_2) c_1,\ (b_1+b_2) c_2,
\\
a_1(c_1+c_2),\ a_2(c_1+c_2),\ b_1(c_1+c_2), \ b_2(c_1+c_2),
\\
(a_1+a_2)(b_1+b_2),\ (a_1+a_2)(c_1+c_2),\ (b_1+b_2)(c_1+c_2),
\\
a_1 b_1 c_1,\ a_1 b_1 c_2,\ a_1 b_2 c_1,\ a_1 b_2 c_2,\
a_2 b_1 c_1,\ a_2 b_1 c_2,\ a_2 b_2 c_1,\ a_2 b_2 c_2, 
\\
(a_1+a_2) b_1 c_1,\ (a_1+a_2) b_1 c_2,\ (a_1+a_2) b_2 c_1,\ (a_1+a_2) b_2 c_2,
\\
a_1(b_1+b_2) c_1,\ a_1(b_1+b_2) c_2,\ a_2 (b_1+b_2) c_1,\ a_2 (b_1+b_2)c_2,
\\
a_1 b_1(c_1+c_2),\ a_1 b_2 (c_1+c_2),\ a_2 b_1 (c_1+c_2),\ a_2 b_2 (c_1+c_2),
\\
(a_1+a_2)(b_1+b_2) c_1,\ (a_1+a_2)(b_1+b_2) c_2,\
(a_1+a_2) b_1 (c_1+c_2),\ (a_1+a_2) b_2 (c_1+c_2),
\\
a_1 (b_1+b_2)(c_1+c_2),\ a_2(b_1+b_2)(c_1+c_2),\
(a_1+a_2)(b_1+b_2)(c_1+c_2).
\end{multline*}
\end{corollary}

Before going to the proof of the more general Theorem \ref{productsums}, we show
below a simple ``nonstandard" argument to directly prove the previous corollary.

\begin{proof}[Proof of Corollary \ref{cor-sumproduct}]
Let $\nu,\mu,\tau$ be infinite hypernatural numbers such that:\footnote
{The existence of such numbers $\nu,\mu,\tau$ is a particular
case of the existence of idempotent ``Folkman numbers." See Remark \ref{existence} below.}
\begin{enumerate}
\item
$\nu\,\ueq\,\mu\,\ueq\,\tau\,\ueq\,\mu+\tau$.
\item
$\nu$ is multiplicatively idempotent, \emph{i.e.}, $\nu\,\ueq\,\nu\cdot{}^*\nu$.
\end{enumerate}

Consider the following numbers in the iterated hyper-extensions of $\N$:
\begin{itemize}
\item
$\alpha_1:=\mu\cdot{}^{***}\nu$ and $\alpha_2:=\tau\cdot{}^{***}\nu$.
\item
$\beta_1:={}^*\mu\cdot{}^{****}\nu$ and $\beta_2:={}^*\tau\cdot{}^{****}\nu$.
\item
$\gamma_1:={}^{**}\mu\cdot{}^{*****}\nu$ and $\gamma_2:={}^{**}\tau\cdot{}^{*****}\nu$.
\end{itemize}
Note that all numbers in 
$\text{FP}(\{\alpha_1<\alpha_2\},\{\beta_1<\beta_2\},\{\gamma_1<\gamma_2\})$
are $u$-equivalent to $\nu$.
The proofs are straightforward; let us see here one example in detail to illustrate the idea:
\begin{multline*}
\alpha_1\cdot(\beta_1+\beta_2)\cdot\gamma_2=
(\mu\cdot{}^{***}\nu)\cdot({}^*\mu\cdot{}^{****}\nu+{}^*\tau\cdot{}^{****}\nu)\cdot({}^{**}\tau\cdot{}^{*****}\nu)=
\\
=\mu\cdot{}^*(\mu+\tau)\cdot{}^{**}\tau\cdot{}^{***}\nu\cdot{}^{****}\nu\cdot{}^{*****}\nu\ \ueq\
\nu\cdot{}^*\nu\cdot{}^{**}\nu\cdot{}^{***}\nu\cdot{}^{****}\nu\cdot{}^{*****}\nu\,\ueq\,\nu.
\end{multline*}

Since all the numbers in 
$\text{FP}(\{\alpha_1<\alpha_2\},\{\beta_1<\beta_2\},\{\gamma_1<\gamma_2\})\subseteq{}^{******}\N$
are $u$-equivalent to each other, by the nonstandard characterization of partition
regular families, we reach the thesis.
\end{proof}

A couple of simple particular cases are the following.

\begin{corollary}\label{cor-3sumproducts}
In every finite coloring of $\N$ there exist monochromatic patterns of the 
following forms: 
\begin{enumerate}
\item
$x,\ y,\ z,\ x+y,\ x\cdot z,\ y\cdot z,\ x\cdot z+y\cdot z$, where $x<y<z$.
\item
$a,\ b,\ x,\ y,\ a+b=xy$, where $x<y<a<b$.
\end{enumerate}
\end{corollary}

\begin{proof}
Consider the monochromatic configuration seen in the previous corollary.
Pattern $(1)$ is obtained by taking $x:=a_1$, $y:=a_2$, and $z:=b_1$;
and pattern $(2)$ is obtained by taking $x:=a_1+a_2$, $y:=b_1$, $a:=a_1b_1$, and $b:=a_2b_1$.
\end{proof}

\begin{remark}
The pattern (2) above was studied in \cite{bo},
where a proof of its monochromaticity is provided that only relies on Ramsey's Theorem. 
This paper also gives an explicit (triple exponential) bound on the minimum length $N=N(r)$
of an initial interval with the property that every $r$-coloring 
contains such a monochromatic pattern.
\end{remark}

Let us finally turn to the proof of the general result.

\begin{proof}[Proof of Theorem \ref{productsums}]
The argument is similar to the one used for Hindman's Theorem.
Let $\alpha$ be an infinite hypernatural number such that:
\begin{enumerate}
\item
$\alpha$ is an ``Folkman number", \emph{i.e.},
for every $\ell$ there exist numbers $\alpha_1<\ldots<\alpha_\ell$ in $\hN$
such that the sum $\sum_{j\in F}\alpha_j\,\ueq\,\alpha$ for every nonempty $F\subseteq\{1,\ldots,\ell\}$.
\item
$\alpha$ is multiplicatively idempotent, \emph{i.e.}, $\alpha\,\ueq\,\alpha\cdot{}^*\alpha$.
\end{enumerate}

(About the existence of such numbers $\alpha$, see Remark \ref{existence} below.)
Given a finite coloring $\N=C_1\cup\ldots\cup C_r$, pick the color $C=C_j$ such that $\alpha\in{}^*C$.
We recursively define nonempty finite sets $A_n\subset\N$ such that:
\begin{itemize}
\item
$X_n:=\text{FP}(\text{FS}(A_1),\ldots,\text{FS}(A_n))\subseteq C$.
\item
$X_n\cdot\alpha\subseteq {}^*C$.
\item
$|A_n|=f(n)$.
\item
$\max X_n<\min A_{n+1}$.
\end{itemize}

Pick hypernatural numbers $\alpha_1<\ldots<\alpha_{f(1)}$ such that
$\sum_{j\in F}\alpha_j\,\ueq\,\alpha$, and hence $\sum_{j\in F}\alpha_j\in{}^*C$, 
for every nonempty $F\subseteq\{1,\ldots,f(1)\}$.
By multiplicative idempotency, 
$(\sum_{j\in F}\alpha_j)\cdot{}^*\alpha\,\ueq\,\alpha\cdot{}^*\alpha\,\ueq\,\alpha$,
and so $(\sum_{j\in F}\alpha_j)\cdot{}^*\alpha\in{}^{**}C$.
By \emph{backward transfer} we obtain the existence of natural numbers
$a_1<\ldots<a_{f(1)}$ such that $\sum_{j\in F}a_j\in C$ and $(\sum_{j\in F}a_j)\cdot\alpha\in{}^*C$
for every nonempty $F\subseteq\{1,\ldots,f(1)\}$.
The set $A_1:=\{a_1<\ldots<a_{f(1)}\}$ satisfies the induction base step, since trivially 
$\text{FP}(\text{FS}(A_1))=\text{FS}(A_1)=\{\sum_{i\in F}a_i\mid\emptyset\ne F\subseteq\{1,\ldots,f(1)\}\}$.

At the inductive step $n+1$, by the hypothesis we know that
$x\cdot\alpha\in{}^*C$ for every $x\in X_n$.
Then $\alpha\in {}^*C(n)$ where $C(n):=C\cap\bigcap_{x\in X_n}C/x$.
Pick infinite hypernatural numbers $\alpha_1<\ldots<\alpha_{f(n+1)}$ such that
$\sum_{j\in F}\alpha_j\,\ueq\,\alpha$,
and hence $\sum_{j\in F}\alpha_j\in{}^*C(n)$, for every nonempty $F\subseteq\{1,\ldots,f(n+1)\}$.
By multiplicative idempotency, it follows that also
$(\sum_{j\in F}\alpha_j)\cdot{}^*\alpha\in{}^{**}C(n)$.
Now apply \emph{backward transfer} to obtain the existence of 
natural numbers $a_1<\ldots<a_{f(n+1)}$,
all larger than $\max X_n$, and such that 
for every nonempty $F\subseteq\{1,\ldots,f(n+1)\}$,
it is $\sum_{j\in F}a_j\in C(n)$, and $(\sum_{j\in F}a_j)\cdot \alpha\in{}^*C(n)$.
If $A_{n+1}:=\{a_1<\ldots<a_{f(n+1)}\}$, this means that 
$\text{FS}(A_{n+1})\subseteq C$, $\text{FS}(A_{n+1})\cdot X_n\subseteq C$,
$\text{FS}(A_{n+1})\cdot\alpha\subseteq{}^*C$, and 
$\text{FS}(A_{n+1})\cdot X_n\cdot\alpha\subseteq{}^{**}C$.
We conclude that $A_{n+1}$ satisfies all desired properties, since
$X_{n+1}=X_n\cup(\text{FS}(A_{n+1})\cdot X_n)$.
\end{proof}

\begin{remark}\label{existence}
An ultrafilter $\U$ on $\N$ is called \emph{Folkman} if it is a ``witness" of
Folkman's Theorem, \emph{i.e.}, if for every $A\in\U$ and for every $\ell\in\N$ there exist
$x_1<\ldots<x_\ell$ such that $\sum_{j\in F}x_j\in A$ for all nonempty $F\subseteq\{1,\ldots,\ell\}$. 
The space $\mathbb{F}\subseteq\beta\N$ of Folkman ultrafilters
is closed and is a bilateral ideal in $(\beta\N,\odot)$, and therefore is a sub-semigroup of $(\mathbb{F},\odot)$. 
So, one can apply Ellis' Lemma and obtain
the existence of a non-principal ultrafilter $\U\in\mathbb{F}$ that is multiplicatively idempotent. 
In nonstandard terms, this means that there
exists an infinite hypernatural number $\alpha\in\hN$ such that:
\begin{enumerate}
\item
For every $\ell$ there exist $\alpha_1<\ldots<\alpha_\ell$
such that the sum $\sum_{j\in F}\alpha_j\,\ueq\,\alpha$ for every nonempty $F\subseteq\{1,\ldots,\ell\}$.
\item
$\alpha\,\ueq\,\alpha\cdot{}^*\alpha$.
\end{enumerate}
\end{remark}

\begin{remark}\label{equivalentBowen}
The following property was recently proved by M. Bowen (see Theorem 3.1 in \cite{bo}):
\begin{itemize}
\item[$(\dagger)$]
\ \emph{For every finite coloring of $\N$ and for every $f:\N\to\N$,
there exists a sequence of finite sets $(S_n)$ with $|S_n|=f(n)$ such that,
for every finite $F\subset\bigcup_{n\in\N}S_n$, the following numbers are monochromatic:}
$$\prod_{n: S_n\cap F\ne\emptyset}\sum_{s\in S_n\cap F} s.$$
\end{itemize}

We observe that this result is equivalent to the property of Theorem \ref{productsums}.
Indeed, given a finite $F\subseteq\bigcup_{n\in\N}A_n$, 
let $\{n_1<\ldots<n_k\}=\{n\in\N\mid A_n\cap F\ne\emptyset\}$,
and let $b_{n_j}:=\sum_{s\in A_{n_j}\cap F}\in\text{FS}(A_{n_j})$.
Then $\prod_{n: A_n\cap F\ne\emptyset}\sum_{s\in A_n\cap F} s=
b_{n_1}\cdot\ldots\cdot b_{n_k}\in\text{FP}(\text{FS}(A_n))_{n=1}^\infty$.
Conversely, assume $(\dagger)$, and let $a\in\text{FP}(\text{FS}(S_n))_{n=1}^\infty$. Then
$a=b_{n_1}\cdot\ldots\cdot b_{n_k}$ where $b_{n_j}:=\sum_{s\in G_j}s$ 
for appropriate nonempty subsets $G_j\subseteq S_{n_j}$ with $n_1<\ldots<n_k$.
If we let $F:=\bigcup_{j=1}^k G_j$, and we assume that the sets $S_{n_j}$ are pairwise disjoint
(this is always possible following Bowen's proof), it is easily verified that 
$a=\prod_{n: S_n\cap F\ne\emptyset}\sum_{s\in S_n\cap F} s$.

The proof of Theorem \ref{productsums} uses a multiplicatively idempotent ultrafilter that
is a witness of Folkman's Theorem. On the other hand,
Bowen's proof uses multiplicatively central sets, \emph{i.e.},
members of multiplicatively idempotent ultrafilter that are minimal,
and the observation that any multiplicatively minimal ultrafilter is a witness of Folkman's Theorem.\footnote
{Indeed, every set in a multiplicatively minimal ultrafilter is multiplicatively
piecewise syndetic, and every multiplicatively piecewise syndetic set contains 
arbitrarily large sets of finite sums.}
\end{remark}

\section{Infinite monochromatic patterns and the Ramsey property}

A series of recent results obtained by B. Kra, J. Moreira, F. Richter, and D. Robertson 
(see \cite{kmrr} and the references therein)
has sparked researchers' interest in the existence of infinite arithmetic monochromatic patterns 
originating from sequences $(a_n)_{n\in\N}$ of natural numbers. 
Historically, the first example of such monochromatic patterns is provided by the 
famous Hindman's Theorem. It is worth noting that, while finite arithmetic patterns
have been extensively studied and a good understanding of their monochromaticity 
properties has been achieved, little is known about infinite monochromatic patterns.

\smallskip
Let $f:\N^2\to\N$ be any function. By a straight application of Ramsey's
Theorem for pairs, it is proved that for every finite coloring
of $\N$ there exists an increasing sequence $(a_n)_{n\in\N}$
such that all elements $f(a_n,a_m)$ for $n<m$ are monochromatic.
If we also require the elements of the sequence to be of the same
color, we obtain the following interesting notion.

\begin{definition}
A function $f:\N^2\to\N$ has the \emph{Ramsey property} if for every finite coloring of $\N$
there exists an increasing sequence $(a_n)_{n\in\N}$ such that
$a_n, f(a_n,a_m)$ are monochromatic for all $n<m$.
\end{definition}

As a particular case of Hindman's Theorem, the sum operation
$\text{Sum}:(n,m)\mapsto n+m$ has the Ramsey property.
This is also true for the product $\text{Prod}:(n,m)\mapsto n\cdot m$
and, more generally, for every associative operation on $\N$.\footnote
{For instance, all functions $(n,m)\mapsto \ell nm+k(n+m)+\frac{k(k-1)}{\ell}$ 
where $\ell$ divides $k(k-1)$ have the Ramsey property.
(These associative operations have been introduced and studied in \cite{d2}.)}

\begin{definition}
Let $f,g:\N^2\to\N$.
The pair $(f,g)$ satisfies the \emph{Ramsey property} if 
for every finite coloring of $\N$ there exists an increasing 
sequence $(a_n)_{n\in\N}$ such that
$a_n, f(a_n,a_m), g(a_n,a_m)$ are monochromatic for all $n<m$.

We say that $(f,g)$ satisfies the \emph{weak Ramsey property} if 
we do not require the elements of the sequence to be of the same color, 
\emph{i.e.}, if for every finite coloring of $\N$ there exists an increasing 
sequence $(a_n)_{n\in\N}$ such that
$f(a_n,a_m), g(a_n,a_m)$ are monochromatic for all $n<m$.
\end{definition}

We observe that the above definition could be naturally generalized to $t$-uples of functions.
However, for simplicity, here we will only focus on pairs.

Below, we will present examples of pairs of functions with the Ramsey property,
as well as negative examples.
In the proofs, we will use special pairs of hypernatural numbers, named ``Ramsey pairs".

As a starting observation, we note that the Ramsey property 
is closely related to Ramsey's Theorem for pairs, that states the partition regularity on $\N^2$
of the following family of infinite patterns:
$$\mathfrak{R}=\big\{\{(a_n,a_m)\mid n<m\}\mid (a_n)_{n\in\N}\ \text{increasing sequence in}\ \N\big\}.$$

The following notion was introduced in \cite{dlmrv}, and revealed useful in proving
several Ramsey results.

\begin{definition}
A pair $(\alpha,\beta)\in\hN^2$ is a \emph{Ramsey pair}
if whenever $(\alpha,\beta)\in\hX$ there exists an increasing sequence
$(a_n)_{n\in\N}$ such that $(a_n,a_m)\in X$ for all $n<m$.
\end{definition}

Note that, by the characterizations of Theorem \ref{PRcharacterizations},
that also holds for families of subsets of $\N^2$, Ramsey pairs exist.

\begin{proposition}\label{Ramseypairs}
If $(\alpha,\beta)$ is a Ramsey pair, then:
\begin{enumerate}
\item
$\alpha<\beta$.\footnote
{In fact a much stronger property holds showing that $\beta$ is
``way larger" than $\alpha$ (see Proposition \ref{Ramseypairproperty} below).}

\item
$\alpha,\beta\in\hN\setminus\N$ are infinite.
\item
$\alpha\ueq\beta$.
\end{enumerate}
\end{proposition}

\begin{proof}
$(1)$. If $\alpha\ge\beta$ then $(\alpha,\beta)\in{}^*(\Delta^-)$
where $\Delta^-=\{(k,h)\in\N^2\mid k\ge h\}$,
and clearly for every increasing sequence $(a_n)_{n\in\N}$
one has $(a_n,a_m)\notin\Delta^-$ for all $n<m$.

$(2)$. By $(1)$, it is enough to show that $\alpha$ is infinite.
If $\alpha=k\in\N$ was finite, then we would have
$(k,\beta)\in\{k\}\times\hN={}^*(\{k\}\times\N)$, 
and clearly there is no increasing sequence $(a_n)_{n\in\N}$
such that $(a_n,a_m)\in\{k\}\times\N$ for all $n<m$.

$(3)$. If it was $\alpha\not\ueq\beta$ then we could pick $A\subseteq\N$ such
that $\alpha\in\hA$ and $\beta\in{}^*(A^c)$.
Then $(\alpha,\beta)\in{}^*(A\times A^c)$; however, for every
increasing sequence $(a_n)_{n\in\N}$ and for every $n\in\N$,
one has that $(a_n,a_{n+1})\in A\times A^c\Rightarrow (a_{n+1},a_{n+2})\notin A\times A^c$.
\end{proof}

By the ultrafilter characterizations seen in Remark \ref{PR-ultrafilters},
that also hold for families of subsets of $\N^2$, 
we see that the partition regularity of the family $\mathfrak{R}$
is equivalent to the existence of PR-witness ultrafilters for $\mathfrak{P}$,
that we name ``Ramsey's witnesses."

\begin{definition}
An ultrafilter $\W$ on $\N^2$ is a \emph{Ramsey's witness}
if for every $X\in\W$ there exists an increasing sequence $(a_n)_{n\in\N}$
such that $(a_n,a_m)\in X$ for all $n<m$.
\end{definition}

It is readily seen that
$(\alpha,\beta)$ is a Ramsey pair if and only if
the generated ultrafilter $\UU_{(\alpha,\beta)}:=\{X\subseteq\N^2\mid (\alpha,\beta)\in\hX\}$
on $\N^2$ is a Ramsey's witness.

\begin{remark}
Note that the properties seen in Proposition \ref{Ramseypairs} correspond to the following: 
\begin{itemize}
\item
\emph{Let $\W$ be a Ramsey's witness. Then:}
\begin{enumerate}
\item
\emph{The upper diagonal
$\Delta^+:=\{(n,m)\in\N^2\mid n<m\}\in\W$.}
\item
\emph{The image ultrafilters $\pi_1(\W)$ and $\pi_2(\W)$ on $\N$ where
$\pi_1:(n,m)\mapsto n$ and $\pi_2:(n,m)\mapsto m$ are the canonical projections, 
are non-principal.}
\item
\emph{$\pi_1(\W)=\pi_2(\W)$.}
\end{enumerate}
\end{itemize}
\end{remark}

\smallskip
The Ramsey properties of functions
can be equivalently reformulated in neat and simple terms.

\begin{theorem}\label{characterizationsRamseyproperty}
Let $f,g:\N^2\to\N$.
\begin{itemize}
\item
The following conditions are equivalent:
\begin{enumerate}
\item
The function $f$ has the Ramsey property.
\item
There exists a Ramsey pair $(\alpha,\beta)\in\hN^2$ such that
$\hf(\alpha,\beta)\ueq\alpha$.
\item
There exists a Ramsey's witness $\W\in\beta\N^2$ such that $f(\W)=\pi_1(\W)$.
\end{enumerate}
\item
The following conditions are equivalent:
\begin{enumerate}
\item
The pair $(f,g)$ has the Ramsey property.
\item
There exists a Ramsey pair $(\alpha,\beta)\in\hN^2$ such that
$\hf(\alpha,\beta)\ueq \hg(\alpha,\beta)\ueq\alpha$.
\item
There exists a Ramsey's witness $\W\in\beta\N^2$ such that
$f(\W)=g(\W)=\pi_1(\W)$.
\end{enumerate}
\item
The following conditions are equivalent:
\begin{enumerate}
\item
The pair $(f,g)$ has the weak Ramsey property.
\item
There exists a Ramsey pair $(\alpha,\beta)\in\hN^2$ such that
$\hf(\alpha,\beta)\ueq \hg(\alpha,\beta)$.
\item
There exists a Ramsey's witness $\W\in\beta\N^2$ such that
$f(\W)=g(\W)$.
\end{enumerate}
\end{itemize}
\end{theorem}

\begin{proof}
First of all, we observe that items $(2)$ and $(3)$ in each of the three groups are trivially equivalent,
since they are the nonstandard and the ultrafilter formulations respectively of the same properties.
For example, in the first group, $\hf(\alpha,\beta)\ueq\alpha$ means that
$\UU_{\hf(\alpha,\beta)}=\UU_\alpha$; then recall that $\UU_{\hf(\alpha,\beta)}=f(\UU_{(\alpha,\beta)})$ and
that $\UU_\alpha=\pi_1(\UU_{(\alpha,\beta)})$.

We observe that a function $f$ satisfies the Ramsey property if and only if the pair $(f,f)$
satisfy the Ramsey property, so the first set of equivalences follows from the second.

We now prove the equivalences in the second group
(the equivalences for the weak Ramsey property are proved in the same way).
Assume first that $\W$ is a Ramsey's witness such that 
$\U:=f(\W)=g(\W)=\pi_1(\W)$. Given a coloring $\N=C_1\cup\ldots\cup C_r$,
let $C=C_i$ be the color such that $C\in\U$.
Then $X:=f^{-1}(C)\cap g^{-1}(C)\cap\pi_1^{-1}(C)\in\W$, and so there exists
an increasing sequence $(a_n)_{n\in\N}$ such that
$(a_n,a_m)\in X$ for all $n<m$; this means that $f(a_n,a_m)\in C$, $g(a_n,a_m)\in C$,
and $\pi_1(a_n,a_m)=a_n\in C$, as desired.

Conversely, assume that the following family is partition regular on $\N$:
$$\mathfrak{P}:=\left\{\bigcup_{n<m}\{a_n, f(a_n,a_m), g(a_n,a_m)\}\ \middle|\ 
(a_n)_{n\in\N}\ \text{increasing sequence in}\ \N\right\}.$$
For $A\subseteq\N$, we say that a sequence of natural numbers $(a_n)_{n\in\N}$
is ``good for $A$" if it is increasing and such that $a_n, f(a_n,a_m), g(a_n,a_m)\in A$
for all $n<m$. By Remark \ref{PR-ultrafilters}, we can pick an ultrafilter $\U$ on $\N$ that is
PR-witness for $\mathfrak{P}$, \emph{i.e.}, such that for every $A\in\U$ there
exists a sequence $(a_n)_{n\in\N}$ which is good for $A$.

Call a subset $X\subseteq\N^2$ ``Ramsey-large" if for every
increasing sequence $(b_n)_{n\in\N}$ of natural numbers
there exist $n<m$ such that $(b_n,b_m)\in X$. Then consider the family
$$\mathcal{F}:=\{\Lambda_A\mid A\in\U\}\cup\mathcal{R}$$
where $\Lambda_A:=\{(a_n,a_m)\mid (a_n)_{n\in\N}\ \text{is good for}\ A,\ n<m\}$,
and $\mathcal{R}$ is the family of Ramsey-large sets.
We claim that the family $\F$ has the finite intersection property.

Notice first that:
\begin{itemize}
\item
\emph{The family $\mathcal{R}$ of Ramsey-large sets
is closed under finite intersections}. 
\end{itemize}
To see this, pick
$X_1,\ldots,X_k\in\mathcal{R}$ and assume for the sake of contradiction
that their intersection $Y:=X_1\cap\ldots\cap X_k\notin\mathcal{R}$. 
This means that there is an increasing sequence $(b_n)_{n\in\N}$ such that
$(b_n,b_m)\in Y^c$ for all $m<n$. Since $Y^c=(X_1)^c\cup\ldots\cup(X_k)^c$,
by Ramsey's Theorem there exists $i$ and a subsequence $(b_{n_s})_{s\in\N}$
such that $(b_{n_s},b_{n_t})\in(X_i)^c$ for all $s<t$. This contradicts the assumption
that $X_i$ is Ramsey-large.

Now let sets $A_1,\ldots,A_m\in\U$ and sets $X_1,\ldots,X_k\in\mathcal{R}$ be given.
We have to show that $\Lambda_{A_1}\cap\ldots\cap\Lambda_{A_m}\cap X_1\cap\ldots\cap X_k\ne\emptyset$.
Note that $\Lambda_{A_1}\cap\ldots\cap\Lambda_{A_m}=\Lambda_B$ where
$B=A_1\cap\ldots\cap A_k\in\U$, and so we can pick an increasing sequence $(a_n)_{n\in\N}$
that is good for $B$. 
Since $Y:=X_1\cap\ldots\cap X_k$ is Ramsey-large, there exist $n<m$
such that the pair $(b_n,b_m)\in Y$,
and hence $(b_n,b_m)\in\Lambda_B\cap Y\ne\emptyset$.

Pick any ultrafilter $\W$ on $\N^2$ that extends the family $\F$.
We observe that for every $A\in\U$, the set 
$\pi_1^{-1}(A)\cap f^{-1}(A)\cap g^{-1}(A)\supseteq\Lambda_A\in\W$,
and hence $A\in\pi_1(\W)\cap f(\W)\cap g(\W)$. 
This shows that $\U=\pi_1(\W)=f(\W)=g(\W)$.
Finally, we observe that $\W$ is a Ramsey's witness. Indeed,
let $X\in\W$; if by contradiction for every increasing
sequence $(a_n)_{n\in\N}$ there exist $n<m$ such that $(a_n,a_m)\in X^c$,
then $X^c\in\mathcal{R}\subseteq\W$, and we would have 
$\emptyset=X\cap X^c\in\W$.
\end{proof}

Below is a useful topological characterization of Ramsey's witnesses.

\begin{proposition}\label{Ramseyclosuretensors}
The space $RW:=\{\W\in\beta\N^2\mid\W\ \text{is a Ramsey's witness}\}$ is the 
closure of the space $\text{TP}=\{\U\otimes\U\mid\U\in\beta\N\setminus\N\}$
of tensor powers of non-principal ultrafilters.
Equivalently, in nonstandard terms, $(\alpha,\beta)\in\hN^2$ is 
a Ramsey pair if and only if the following condition holds:
\begin{itemize}
\item
If $(\alpha,\beta)\in\hX$ then there exists an infinite $\gamma\in\hN$ such that $(\gamma,{}^*\gamma)\in{}^{**}X$.
\end{itemize}
\end{proposition}

\begin{proof}
Let us see first that $\text{TP}\subseteq\text{RW}$. Let $X\in\U\otimes\U$ where $\U\in\beta\N$ is non-principal.
This means that $\widehat{X}:=\{n\in\N\mid X_n\in\U\}\in\U$, where $X_n:=\{m\in\N\mid (n,m)\in X\}$
is the vertical $n$-fiber of $X$. Pick $a_1\in \widehat{X}$. Then $X_{a_1}\in\U$
and we can pick $a_2\in X_{a_1}\cap\widehat{X}\cap(a_1,+\infty)\in\U$. 
Then $(a_1,a_2)\in X$ and $X_{a_2}\in\U$. 
Pick $a_3\in X_{a_1}\cap X_{a_2}\cap\widehat{X}\cap(a_2,+\infty)\in\U$.
Then $(a_1,a_3)\in X$, $(a_2,a_3)\in X$, and $X_{a_3}\in\U$. By inductively iterating the process,
we find an increasing sequence $(a_n)_{n\in\N}$ such that $(a_n,a_m)\in X$ for all $n<m$,
as desired.\footnote
{This precise argument is used for a well-known ultrafilter proof of Ramsey's Theorem for pairs
(see, \emph{e.g.}, Problem 7.5.1 of \cite{je}), and it corresponds to our nonstandard proof 
of Theorem \ref{RamseyNS}.}

The space $\text{RW}\subseteq\beta\N^2$ is closed. If $\W\notin\text{RW}$
then we can pick a set $X\in\W$ that does for every increasing sequence $(a_n)_{n\in\N}$
there exist $n<m$ such that $(a_n,a_m)\notin X$.
Then the open set $\mathcal{O}_X=\{\mathcal{Z}\in\beta\N^2\mid X\in\mathcal{Z}\}$
is disjoint from $\text{RW}$.\footnote
{This is a particular case of the following general fact:
Suppose $\F\subseteq\PP(I)$ is any nonempty family of subsets of $I$;
then the space of all ultrafilters $\U$ on $I$ such that $\U\subseteq\F$ is closed.
(In our case, $\F\subseteq\PP(\N^2)$ is the family of sets $X$ that satisfy
the ``Ramsey property" stating the existence
of an increasing sequence $(a_n)_{n\in\N}$ such that $(a_n,a_m)\in X$ for all $n<m$.) 
Note that such a space of ultrafilters is nonempty if and only if $\F$ is partition regular,
by Theorem \ref{PRcharacterizations}.}

Let $\mathcal{O}_X$ be a basic open neighborhood of an ultrafilter $\W\in\text{RW}$,
\emph{i.e.}, $X\in\W$. Pick an increasing sequence $(a_n)_{n\in\N}$ such
that $(a_n,a_m)\in X$ for all $n\in\N$. The family $\F:=\{A_n\mid n\in\N\}$
where $A_n:=\{a_m\mid m>n\}$ has the finite intersection property, and so 
it can be extended to an ultrafilter $\U$ on $\N$. Clearly $\U$ is non-principal
because $\bigcap_{n\in\N}A_n=\emptyset$. Finally, we observe that
$X\in\U\otimes\U$, and hence $\U\otimes\U\in\mathcal{O}_A\cap\text{TP}\ne\emptyset$.
To see this, note that $X_{a_n}\in\U$ for every $n\in\N$, 
since $A_n\subseteq X_{a_n}$; then $\widehat{X}=\{k\in\N\mid X_k\in\U\}\in\U$,
since $\{a_n\mid n\in\N\}\subseteq \widehat{X}$.

The nonstandard characterization directly follows by recalling that
$(\alpha,\beta)\in\hN^2$ is a Ramsey pair if and only if the generated ultrafilter $\UU_{(\alpha,\beta)}$
is a Ramsey's witness, and by observing that
$X\in\U\otimes\U$ belongs to a tensor power of a non-principal ultrafilter $\U=\UU_\gamma$
if and only if $(\gamma,{}^*\gamma)\in{}^{**}X$ for a suitable infinite $\gamma\in\hN$.
\end{proof}

\subsection{Positive examples}

We already mentioned that, as a corollary of Hindman's Theorem for semigroups,
the sum operation $\text{Sum}:(n,m)\mapsto n+m$, 
the product operation $\text{Prod}:(n,m)\mapsto n\cdot m$, along with all
other associative operations on $\N$, have the Ramsey property.
Below are two examples of non-associative operations that also satisfy that property.

\begin{theorem}\label{differenceRamsey}
The difference function $\text{Diff}:(n,m)\mapsto m-n$ (defined for $n<m$) 
has the Ramsey property, \emph{i.e.},
for every finite coloring there exists an increasing sequence $(a_n)_{n\in\N}$ such
that elements $a_n, a_m-a_n$ are monochromatic for all $n<m$.
\end{theorem}

\begin{proof}
This is a direct consequence of Hindman's Theorem.
Indeed, pick an increasing sequence $(x_n)_{n\in\N}$ such that
the set of finite sums $\text{FS}(x_n)_{n=1}^\infty$ is monochromatic, and let
$a_n:=\sum_{i=1}^n x_i$. Then clearly the increasing
sequence $(a_n)_{n\in\N}$ has the desired property.
\end{proof}

Below, we present an alternative proof that requires neither idempotent ultrafilters 
nor Hindman's theorem, but relies solely on any non-principal ultrafilter 
(or any infinite hypernatural number).
We present both the nonstandard proof and the one based on ultrafilters.

\begin{proof}[Nonstandard proof]
Pick any infinite $\nu\in\hN$, and let $(\alpha,\beta):=({}^{*}\nu-\nu, {}^{**}\nu-\nu)$.
Note that $\text{Diff}(\alpha,\beta)={}^{**}\nu-{}^*\nu={}^*({}^*\nu-\nu)\ueq{}^*\nu-\nu=\alpha$.
To see that that $(\alpha,\beta)$ is a Ramsey pair, 
we use the characterization of Proposition \ref{Ramseyclosuretensors}.
If $({}^{*}\nu-\nu, {}^{**}\nu-\nu)\in{}^{***}X$,
then $\nu\in{}^*\{n\in\N\mid (\nu-n,{}^*\nu-n)\in{}^{**}X\}$, and so we can pick $n_0\in\N$ such
that $(\nu-n_0,{}^*\nu-n_0)\in{}^{**}X$. 
Finally, note that $(\nu-n_0,{}^*\nu-n_0)=(\gamma,{}^*\gamma)$ where
$\gamma:=\nu-n_0$ is infinite.
\end{proof}

\begin{proof}[Ultrafilter proof]
We will use the characterization of Theorem \ref{characterizationsRamseyproperty},
and show that there exists a Ramsey's witness $\W$ such that $\text{Diff}(\W)=\pi_1(\W)$.

Let $f:\N^3\to\N^2$ be the function where $(x,y,z)\mapsto (y-x, z-x)$ if $x<y$ and $x<z$, 
and $f(x,y,z)=1$ otherwise.
Pick any non-principal ultrafilter $\U$ on $\N$,
and let $\W:=f(\U\otimes\U\otimes\U)$ be the image ultrafilter of the triple tensor power of $\U$
under the function $f$. 

We show that $\W$ belongs to the closure of the 
space of non-principal tensor powers $\text{TP}$, and hence it is a Ramsey's witness.
To this end, take any element $X\in\W$. Then $f^{-1}(X)=\{(n,m,k)\in\N^3\mid (m-n,k-n)\in X\}\in\U\otimes\U\otimes\U$,
\emph{i.e.}, $\Lambda:=\{n\in\N\mid\{(m,k)\mid (m-n,k-n)\in X\}\in\U\otimes\U\}\in\U$.
Pick $n_0\in\Lambda$; then $X-(n_0,n_0)=\{(m,k)\mid(m-n_0,k-n_0)\in X\}\in\U\otimes\U$,
and hence $X\in\V\otimes\V$ where $\V:=n_0\oplus\U$ is non-principal.

We are left to show that the image ultrafilters $\text{Diff}(\W)=\pi_1(\W)$.
This is proven by the following properties, 
which are one equivalent to the next for any given $A\subseteq\N$:
\begin{itemize}
\item
$A\in\text{Diff}(\W)=\text{Diff}(f(\U\otimes\U\otimes\U))$.
\item
$f^{-1}(\text{Diff}^{-1}(A))=
\{(n,m,k)\in\N^3\mid k-m\in A\}=
\N\times\text{Diff}^{-1}(A)\in\U\otimes\U\otimes\U$.
\item
$\text{Diff}^{-1}(A)\in\U\otimes\U$.
\item
$\text{Diff}^{-1}(A)\times\N=\{(n,m,k)\in\N^3\mid m-n\in A\}=f^{-1}(\pi_1^{-1}(A))\in\U\otimes\U\otimes\U$.
\item
$A\in\pi_1(f(\U\otimes\U\otimes\U))=\pi_1(\W)$.
\end{itemize}
\end{proof}

\begin{theorem}\label{ourexponentiation}
The exponential function $\text{Exp}:(n,m)\mapsto m^n$ has the Ramsey property,
\emph{i.e.}, for every finite coloring there exists an increasing sequence $(a_n)_{n\in\N}$ such
that elements $a_n, (a_m)^{a_n}$ are monochromatic for all $n<m$.
\end{theorem}

\begin{proof}
This is a corollary of the following Hindman-like result for exponentiation 
that was proved in \cite{dra}:
\begin{itemize}
\item
\emph{For every finite coloring of $\N$ there exists an increasing sequence
$(a_n)_{n\in\N}$ such that the following towers of exponentiations 
are monochromatic for all $n_1>\ldots>n_k$.\footnote
{In fact, all possible finite exponentiations where parentheses are put in any meaningful way
belong to the same color (see \cite{dra} for details.)}}
$${a_{n_1}}^ {  {a_{n_2}}^ {  {a_{n_3}}^ {\iddots^  {a_{n_k}} } } }$$
\end{itemize}
\end{proof}

Let us now turn to pairs of functions.
We will see below that the pair given by the sum and product functions does not satisfy the 
weak Ramsey property where the elements of the sequence are not assumed 
to be of the same color (see Theorem \ref{Hindmansumproducts}) . However, quite surpringly, 
the pair given by sum and division has the Ramsey property.

\begin{theorem}
Let $\text{Div}:\N^2\to\N$ be the function where $(n,m)\mapsto\frac{m}{n}$ if $m$ is a multiple of $n$,
and $\text{Div}(n,m)=1$ otherwise. Then the pair $(\text{Sum},\text{Div})$ 
satisfies the Ramsey property, \emph{i.e.}, for every finite coloring of $\N$ there
exists an increasing sequence $(a_n)_{n\in\N}$ such that
element $a_n, a_n+a_m, \frac{a_m}{a_n}$ are monochromatic for all $n<m$.\footnote
{Clearly, in particular $a_m$ is a multiple of $a_n$ for all $n<m$.}
\end{theorem}

\begin{proof}
This is a straight corollary of the following recent extension of Hindman's Theorem, proved in \cite{dlmrv}:
\begin{itemize}
\item
\emph{For every finite coloring of $\N$ there exists an increasing sequence
$(a_n)_{n\in\N}$ such that the following elements
are monochromatic for all $n_1<\ldots<n_\ell<n_{\ell+1}<\ldots<n_k$:}
$$a_{n_1}+\ldots+a_{n_\ell},\ \frac{a_{n_{\ell+1}}+\ldots+a_{n_k}}{a_{n_1}+\ldots+a_{n_\ell}}.$$
\end{itemize}
\end{proof}

\subsection{Negative examples}

The following negative result proved by N. Hindman in \cite{hi}
shows the impossibility of infinite monochromatic patterns that combine sums and products
originated by a same sequence.

\begin{theorem}\label{Hindmansumproducts}
The pair of functions $(\text{Sum}, \text{Prod})$ does not satisfy the weak Ramsey property.
That is, there exists a finite coloring of $\N$ such that no increasing sequence $(a_n)_{n\in\N}$
has the property that all elements $a_n+a_m, a_n\cdot a_m$ for $n<m$ are monochromatic.
\end{theorem}

About exponentiations and products, recall the following result that 
was proved by J. Sahasrabudhe in \cite{sa}:
\begin{itemize}
\item
\emph{For every $\ell\in\N$ and for every finite coloring of $\N$
there exist elements $a_1<\ldots<a_\ell$ such that
the following finite products and finite towers of exponentiation 
for $1\le n_1<\ldots<n_k\le\ell$ are all monochromatic:}
$$a_{n_1}\cdot\ldots\cdot a_{n_k}, \quad {a_{n_1}}^ {  {a_{n_2}}^ {  {a_{n_3}}^ {\iddots^  {a_{n_k}} } } }$$
\end{itemize}

We stress that the above exponentiations are considered in reverse
order with respect to the exponential function $\text{Exp}:(n,m)\mapsto m^n$
as considered in Theorem \ref{ourexponentiation}.
Although there exist arbitrarily large finite monochromatic patterns,
this reverse order prevents the infinitary Ramsey property to be satisfied.

\begin{theorem}
The pair of functions $(\text{Prod},\text{Exp}')$ 
where $(n,m)\mapsto n^m$ does not satisfy the weak Ramsey property.
That is, there exists a finite coloring of $\N$
such that no increasing sequence $(a_n)_{n\in\N}$
has the property that all elements $a_n\cdot a_m, (a_n)^{a_m}$ for $n<m$ are monochromatic.\footnote
{This is Example 5.23 of \cite{dlmrv}.}
\end{theorem}

\begin{proof}
Assume for the sake of contradiction that there exists a Ramsey pair
$(\alpha,\beta)$ such that $\alpha\beta\ueq\alpha^\beta$.
Let $\varphi,\psi:\N\to\N$ be the functions where $\varphi:n\mapsto\lfloor\log_2(n)\rfloor$,
and $\psi:n\mapsto\lfloor\log_2(\log_2(n))\rfloor$.\footnote
{For simplicity, in the following we will omit stars when denoting hyper-extensions of functions.}
Note that 
$\psi(\alpha\beta)=\lfloor\log_2(\log_2(\beta))+\varepsilon\rfloor$ where
$\varepsilon:=\log_2(1+\frac{\log_2(\alpha)}{\log_2(\beta)})<1$,
and so 
\begin{itemize}
\item
$\psi(\alpha\beta)=\lfloor\log_2(\log_2(\beta))\rfloor+k$ where $k\in\{0,1\}$.
\end{itemize}
Note also that $\psi(\alpha^\beta)=
\log_2(\beta)+\log_2(\log_2(\alpha))+r$
where $0\le r<1$; then 
$\varphi(\psi(\alpha^\beta))=\lfloor\log_2(\log_2(\beta))+\delta\rfloor$
where $\delta:=\log_2(1+\frac{\log_2(\log_2(\alpha))+r}{\log_2(\beta)})<1$,
and so 
\begin{itemize}
\item
$\varphi(\psi(\alpha^\beta))=\lfloor\log_2(\log_2(\beta))\rfloor+h$ where $h\in\{0,1\}$.
\end{itemize}
Now let $\nu:=\psi(\alpha\beta)$ and $\mu:=\psi(\alpha^\beta)$.
Since $\alpha\beta\ueq\alpha^\beta$ we have that $\nu\ueq\mu$.
Besides, if $f$ is the function $f:n\mapsto n+k-h$, then
$f(\varphi(\mu))=\nu\ueq\mu$ and hence, by Theorem \ref{u-equivalenceproperties} (5),
it must be $f(\varphi(\mu))=\mu$, and hence $\mu$ and $\varphi(\mu)$ are at finite distance.
This is not possible because $\mu$ is infinite, and so $\frac{\mu}{\varphi(\mu)}\ge\frac{\mu}{\log_2(\mu)}$ is infinite.
\end{proof}

\begin{theorem}
The function $\text{Exp}'$ does not
satisfy the Ramsey property. That is, there exists a finite coloring of $\N$
such that no increasing sequence $(a_n)_{n\in\N}$
has the property that all elements $a_n, (a_n)^{a_m}$ for $n<m$ are monochromatic.\footnote
{This follows from Lemma 23 of \cite{sa}.}
\end{theorem}

\begin{proof}
By using the same arguments as in the proof of the previous theorem,
a nonstandard proof is obtained by showing that the existence of a Ramsey pair $(\alpha,\beta)$ 
such that $\beta\ueq \alpha^\beta$ leads to a contradiction.
\end{proof}

In order to demonstrate further negative results, the following property
of Ramsey's witnesses proved crucial.

\begin{proposition}\label{Ramseypairproperty}
Let $f,g:\N\to\N$. If $\W$ is a Ramsey's witness, then
$$X:=\{(n,m)\in\N^2\mid g(n)\ne g(m)\}\in\W\ \Rightarrow\ 
Y:=\{(n,m)\in\N^2\mid f(n)<g(m)\}\in\W.$$
Equivalently, if $(\alpha,\beta)\in\hN^2$ is a Ramsey pair, then 
$\hg(\alpha)\ne \hg(\beta)\Rightarrow \hf(\alpha)<\hg(\beta)$.\footnote
{See \cite[Prop. 2.13]{dlmmr}.}
\end{proposition}

\begin{proof}
Recall the following properties of tensor products and ultrafilter images:
\begin{itemize}
\item
For every principal $\V\in\beta\N$, the diagonal $\Delta:=\{(n,n)\mid n\in\N\}\in\V\otimes\V$.
\item
For all $\V,\V'\in\beta\N$, if $\V'$ is non-principal then the upper-diagonal
$\Delta^+:=\{(n,m)\mid n<m\}\in\V\otimes\V'$ 
\item
For all $\V,\V'\in\beta\N$, the ultrafilter image $(f,g)(\V,\V')=f(\V)\otimes g(\V')$.\footnote
{$(f,g):\N^2\to\N^2$ denotes the function $(n,m)\mapsto(f(n),g(m))$.}
\end{itemize}
Assume for the sake of contradiction that $X\in\W$ but $Y\notin\W$;
then $X\cap Y^c\in\W$. By Proposition \ref{Ramseyclosuretensors},
there exists a non-principal $\U\in\beta\N$ such that $X\cap Y^c\in\U\otimes\U$,
and hence both $X\in\U\otimes\U$ and $Y^c\in\U\otimes\U$.
We observe that $g(\U)$ is non-principal, otherwise
the diagonal $\Delta\in g(\U)\otimes g(\U)=(g,g)(\U\otimes\U)$,
and hence we would have $(g,g)^{-1}(\Delta)=\{(n,m)\in\N^2\mid g(n)=g(m)\}\in\U\otimes\U$,
against the assumption $X\in\U\otimes\U$.
Finally, observe that, since $g(\U)$ is non-principal,
the upper diagonal $\Delta^+\in f(\U)\otimes g(\U)=(f,g)(\U\otimes\U)$, and hence
$Y=(f,g)^{-1}(\Delta^+)\in\U\otimes\U$, against the assumption $Y^c\in\U\otimes\U$.

To derive the nonstandard characterization, recall that
$(\alpha,\beta)$ is a Ramsey pair if and only if 
the generated ultrafilter $\W:=\UU_{(\alpha,\beta)}$ is a Ramsey's witness.
Then observe that $X\in\W\Leftrightarrow (\alpha,\beta)\in\hX\Leftrightarrow \hg(\alpha)\ne \hg(\beta)$,
and that $Y\in\W\Leftrightarrow (\alpha,\beta)\in{}^*Y\Leftrightarrow \hf(\alpha)<\hg(\beta)$.
\end{proof}

I do not know whether the above property characterizes
Ramsey pairs.

\smallskip
Before proceeding, we need to introduce a few notions.
Let $p$ be a prime number. We will consider the following functions:
\begin{itemize}
\item
The $p$-adic valuation $v_p:\N\to\N_0$ where $v_p(n)$ is the exponent 
of $p$ in the prime factorization of $n$. 
\item
$\text{smod}_p:\N\to\{1,2,\ldots,p-1\}$
where $\text{smod}_p(n)$ is congruent to $\frac{n}{p^{v_p(n)}}$ modulo $p$.
\end{itemize}

We observe that if we write $n$ in base $p$, and we count digit positions from right to left
starting with position $0$ as the last digit on the right, then $v_p(n)$ is the position of 
the first nonzero digit, and $\text{smod}_p(n)$ is that digit.
Clearly, every number $n$ can be written in the form
$n=p^{v_p(n)}n'$ where $n'\equiv \text{smod}_p(n)\not\equiv 0\mod p$.

We will also consider the following:
\begin{itemize}
\item
For $k\in\N_0$, let $d_{p,k}:\N\to\{0,1,\ldots,p-1\}$ be the function where
$d_{p,k}(n)$ is the coefficient of $p^k$ in the base $p$ expansion of $n$.
\end{itemize}
In other words, $d_{p,k}(n)$ is the digit that appears in the $k$-th position from the 
right of $n$ written in base $p$ (recall that we start with position $0$ as the last digit on the right).\footnote
{\emph{E.g.}, $d_{3,3}(139)=2$ because
$139$ written in base $3$ is the string $12011$; \emph{i.e.},
$139=1\cdot 3^4+2\cdot 3^3+0\cdot 3^2+1\cdot 3+1$ .}

For simplicity in what follow we will write $v_p$, $\text{smod}_p$, and $d_{p,k}$
also to denote the hyper-extensions ${}^*v_p$, ${}^*\text{smod}_p$, and ${}^*d_{p,k}$,
respectively. 

Note that, since the functions $\text{smod}_p$ and $d_{p,k}$ have a finite range,
$\xi\ueq\eta$ implies that $\text{smod}_p(\xi)=\text{smod}_p(\eta)$
and $d_{p,k}(\xi)=d_{k,p}(\eta)$. We will use the following property, that is easily verified.
\begin{itemize}
\item[$(\dagger)$]
\ \ For all $x,y\in\N$, it is $d_{p,k}(p^k\cdot x+y)\equiv x+d_k(y)\mod p$.
\end{itemize}
Note that, by \emph{transfer}, the same property holds for all $\xi,\eta\in\hN$.

\smallskip
The following general result was recently proved by nonstandard methods.

\begin{theorem}
Let $f(x,y)=a x^n+P(y)$ where $a\in\Z$, $n\in\N$, and $P(y)\in\Z[y]$ is
a polynomial with $P(0)=0$. Then $f$ satisfies the Ramsey property
if and only if $f$ is a multiple of $\text{Sum}$ or of $\text{Diff}$.
\end{theorem}

\begin{proof}
This is the combination of Corollary 5.10 and Proposition 5.11 of \cite{dlmrv}.
\end{proof}

For simplicity, we will not provide a complete proof here, 
but will simply consider two relevant special cases in detail; 
this should already be sufficient to illustrate how the nonstandard technique is used.

The first special case reported below provides a negative answer to the shiftless 
version of an open question posed in \cite[Remark 3.15]{kmrr}.\footnote
{This problem was also posed as Question \# 2 in the problem session
of the conference ``Perspectives on Ergodic Theory and Its Interactions"
celebrating the work and impact of Vitaly Bergelson, 
held in Warsaw from June 23-27, 2025.}

\begin{theorem}\label{2n+m}
The function $f:\N^2\to\N$ where $f(n,m)=2n+m$ does not satisfy the
Ramsey property, \emph{i.e.}, 
there exists a finite coloring of $\N$ 
such that for no increasing sequence $(a_n)_{n\in\N}$ the 
elements $a_n, 2a_n+a_m$ for $n<m$ are all monochromatic.
\end{theorem}

\begin{proof}
Assume for the sake of contradiction that the above pattern is partition regular.
Then there exists a Ramsey pair $(\alpha,\beta)$ such that $\alpha\ueq 2\alpha+\beta$.
Let $v=v_5:\N\to\N_0$ be the $5$-adic valuation. Write $\alpha=5^{v(\alpha)}\alpha'$ and
$\beta=5^{v(\beta)}\beta'$ where $\alpha',\beta'$ are not divisible by $5$.
Note that $\alpha'\ueq\beta'$, and hence $\alpha'\equiv\beta'\equiv i\mod 5$ 
where $i=\text{smod}(\alpha)=\text{smod}(\beta)$.

Now observe that $v(\alpha)=v(\beta)=\gamma$ is not possible, as otherwise
we would have $5^\gamma\alpha'\ueq 2\cdot 5^\gamma\alpha'+5^\gamma\beta'=5^\gamma(2\alpha'+\beta')$.
But then $\alpha'\equiv 2\alpha'+\beta' \mod 5$, and hence $i\equiv 3 i\mod 5$,
contradicting $i\ne 0$.

Since $v(\alpha)\ne v(\beta)$, by Proposition \ref{Ramseypairproperty} it must be 
$v(\alpha)<v(\beta)$. This also leads to a contradiction because then we would have
$2\alpha+\beta=5^{v(\alpha)}\gamma'$ where 
$\gamma'=2\alpha'+5^{v(\beta)-v(\alpha)}\beta'\equiv 2i\mod 5$.
Since $\alpha=5^{v(\alpha)}\alpha'$ where $\alpha'\equiv i\mod 5$,
it would follow $2i\equiv i\mod 5$, contradicting $i\ne 0$.
\end{proof}

The second example answers in the negative a
problem posed in \cite[Remark 3.17]{kmrr}, and that was also
proposed in a first version of \cite{ac}.

\begin{theorem}\label{n^2+m}
The function $f:\N^2\to\N$ where $f(n,m)=n^2+m$ does not satisfy the
Ramsey property, \emph{i.e.}, there exists a finite coloring of $\N$ 
such that for no increasing sequence $(a_n)_{n\in\N}$ the 
elements $a_n, a_n^2+a_m$ for $n<m$ are all monochromatic.
\end{theorem}

\begin{proof}
As above, we proceed by contradiction, and suppose 
that there exists a Ramsey pair $(\alpha,\beta)$ such that $\alpha\ueq\alpha^2+\beta$.
Let $v=v_3$ be the $3$-adic valuation, let $\text{smod}=\text{smod}_3$,
and for $k\in\N_0$ let $d_k=d_{3,k}$.

Write $\alpha=3^{v(\alpha)}\alpha'$ and
$\beta=3^{v(\beta)}\beta'$ where $\alpha',\beta'$ are not divisible by $3$.
Note first that $\alpha\equiv \alpha^2+\beta\equiv \alpha^2+\alpha\mod 3$
implies that $\alpha\equiv 0$, and hence also $\beta\equiv 0\mod 3$.
This shows that $v(\alpha),v(\beta)\ge 1$.

Now suppose $v_3(\alpha)=v_3(\beta)=\gamma\ge 1$.
Let $\varphi:\N\to\{0,1,2\}$ be the function where $\varphi(n)=d_{2 v(n)}(n)$,
\emph{i.e.}, $\varphi(n)$ is the coefficient of $3^{2 v(n)}$ in the base $3$ expansion of $n$.
We observe that $\alpha^2+\beta=3^{2\gamma}\cdot\alpha'^2+3^\gamma\cdot\beta'=
3^\gamma(3^\gamma\cdot\alpha'^2+\beta')$,
and so $v(\alpha^2+\beta)=v(\alpha)=\gamma$.
By using $(\dagger)$, we have that 
$$\varphi(\alpha^2+\beta)=d_{2\gamma}(\alpha^2+\beta)=d_{2\gamma}(3^{2\gamma}\cdot\alpha'^2+\beta)\equiv
\alpha'^2+d_{2\gamma}(\beta)\mod 3.$$
On the other hand, $\alpha^2+\beta\ueq\alpha\ueq\beta$ implies 
$\varphi(\alpha^2+\beta)=\varphi(\beta)=d_{2\gamma}(\beta)$, and so 
$\alpha'^2\equiv 0\mod 3$, contradicting $\alpha'\not\equiv 0\mod 3$.

Since $v(\alpha)\ne v(\beta)$, Proposition \ref{Ramseypairproperty} guarantees that
$2 v(\alpha)<v(\beta)$, and also in this case we reach a contradiction.
Indeed, $\alpha^2+\beta=3^{2 v(\alpha)}(\alpha'^2+3^{v(\beta)-2 v(\alpha)}\beta')$,
and so $v(\alpha^2+\beta)=2 v(\alpha)$. On the other hand,
$\alpha^2+\beta\ueq\alpha$ implies that $v(\alpha^2+\beta)\ueq v(\alpha)$.
Finally, by Theorem \ref{u-equivalenceproperties}, we know that
$v(\alpha)\ueq 2 v(\alpha)$ implies $v(\alpha)=2 v(\alpha)$,
contradicting $v(\alpha)\ne 0$.
\end{proof}

\subsection{Open problems}

We close this section by posing three of the main open problems
about the Ramsey property for pairs of functions.

\begin{itemize}
\item
\emph{Does the pair of functions $(\text{Sum}, \text{Diff})$ satisfy the Ramsey property?
That is, is it true that for every finite partition of $\N$
there exists an increasing sequence $(a_n)_{n\in\N}$ such that
elements $a_n+a_m, a_m-a_n$ for $n<m$ are all monochromatic?}

\smallskip
\item
\emph{Does the pair of functions $(\text{Diff}, \text{Prod})$ satisfy the Ramsey property?}

\smallskip
\item
\emph{Does the pair of functions $(\text{Prod}, \text{Exp})$ satisfy the Ramsey property?}\footnote
{Recall that $\text{Exp}:(n,m)\mapsto m^n$.}
\end{itemize}

\section{Ultrapower construction of $\hN$}\label{sec-ultrapower}

In this final section we will present a model of the hypernatural numbers $\hN$
constructed as an appropriate ultrapower of $\N$.

\begin{definition}
Let $\U$ be an ultrafilter on a set $I$, and let $X$ be a nonempty set. 
The \emph{ultrapower} of $X$ modulo $\U$ is the quotient set
$$X^I\!/\U:=\text{Fun}(I,X)/\equiv_\U$$
where two functions $f,g:I\to X$ are ``$\U$-equivalent" if
they agree on a set of indexes that belongs to $\U$, \emph{i.e.},
$$f\equiv_\U g\Longleftrightarrow\{i\in I\mid f(i)=g(i)\}\in\U.$$
\end{definition}

Note that the relation $\equiv_\U$ is reflexive because $I\in\U$;
besides, $\equiv_\U$ is trivially symmetric.
As for transitivity, observe that if $\Lambda:=\{i\in I\mid f(i)=g(i)\}\in\U$
and $\Gamma:=\{i\in I\mid g(i)=h(i)\}\in\U$, then also the
set $\Theta:=\{i\in I\mid f(i)=h(i)\}\in\U$, because $\Theta$ is
a superset of $\Lambda\cap\Gamma\in\U$.

\begin{remark}
While the properties of a filter are enough to prove that $\equiv_\U$ is
an equivalence relation, the use of an ultrafilter is necessary to
guarantee that the resulting ultrapower satisfies the \emph{transfer principle}.

Ultrapowers have strong properties that needs the formalism
of mathematical logic to be precisely formulated.
The fundamental
\emph{\L os' Theorem} states that every ultrapower $X^I\!/\U$ satisfies
the same ``elementary" properties as $X$, where a
property is ``elementary" if it can be formalized as a formula
in an appropriate first-order language. More precisely,
in the case of ultrapowers one can take the richest possible
first-order language that contains one symbol for each element, 
each function, and each relation on the starting set $X$.
\end{remark}

Below we will focus on the natural numbers and
construct a set of hypernatural numbers $\hN$ as an appropriate
ultrapower of $\N$. However, we remark that the given definitions 
and properties also apply to ultrapowers of any set $X$.

\smallskip
Let us fix an ultrafilter $\U$ on a set of indexes $I$,
and denote by $\hN$ the corresponding ultrapower of the natural numbers, \emph{i.e.},
$$\hN=\N^I\!/\U:=\{[f]_\U\mid f:I\to\N\}$$ 

\begin{definition}\label{def-extension}
For every $A\subseteq\N^d$, the \emph{hyper-extension} $\hA\subseteq\hN^d$ 
is defined by setting for all $f_1,\ldots,f_d:I\to\N$:
$$\hA:=\left\{([f_1]_\U,\ldots,[f_d]_\U)\in\hN^d\ \middle|\  
\{i\in I\mid f_s(i)\in A\}\in\U\ \text{for every}\ s=1,\ldots,d\right\}.$$
For every function $F:A\to\N$ where $A\subseteq\N^d$, the \emph{hyperextension}
${}^*F:\hA\to\hN$ is the function defined by setting for every 
$([f_1]_\U,\ldots,[f_d]_\U)\in\hA$:
$${}^*F([f_1]_\U,\ldots,[f_d]_\U)=[F\circ(f_1,\ldots,f_d)]_\U,$$
where $F\circ(f_1,\ldots,f_d):I\to\N$ is the function
$i\mapsto F(f_1(i),\ldots,f_d(i))$.
\end{definition}

We observe that the above definitions are well-posed.
Indeed, if $[f_s]_\U=[f'_s]_\U$ for $s=1,\ldots,d$, then
the sets $\Lambda_s:=\{i\in I\mid f_s(i)=f'_s(i)\}\in\U$, and so
$\Lambda:=\bigcap_{s=1}^d\Lambda_s\in\U$.
In consequence, the following properties are equivalent to each other:
\begin{itemize}
\item
$([f_1]_\U,\ldots,[f_d]_\U)\in\hA$.
\item
$\{i\in I\mid f_s(i)\in A\}\in\U$ for every $s=1,\ldots,d$.
\item
$\{i\in I\mid f_s(i)\in A\}\cap\Lambda\in\U$ for every $s=1,\ldots,d$.
\item
$\{i\in I\mid f'_s(i)\in A\}\in\U$ for every $s=1,\ldots,d$.
\item
$([f'_1]_\U,\ldots,[f'_d]_\U)\in\hA$.
\end{itemize}
Similarly, it also proved that if $[f_s]_\U=[f'_s]_\U$ for every $s=1,\ldots,d$,
then $[F\circ(f_1,\ldots,f_d)]_\U=[F\circ(f'_s,\ldots,f'_d)]_\U$.

\begin{theorem}
Let $I$ be the set of all finite nonempty subsets of $\PP(\N)$,
and let $\U$ be any ultrafilter on $I$ that extends
the family $\F:=\{\check{A}\mid A\subseteq\N\}$, where
$\check{A}=\{i\in I\mid A\in i\}$.
Then the ultrapower $\hN:=\N^I\!/\U$ satisfies the three principles
of extension, transfer, and $\mathfrak{c}$-enlargement.
\end{theorem}

\begin{proof}
We first observe that the family $\F$ has the finite intersection property,
and so there actually are ultrafilters $\U$ extending $\F$.
This is easily seen by noticing that for all $A_1,\ldots,A_k\subseteq\N$,
the index $i:=\{A_1,\ldots,A_k\}\in\bigcap_{s=1}\check{A}_s$.

\smallskip
The \emph{extension} principle is directly the above Definition \ref{def-extension}.
Let us verify that $\mathfrak{c}$-\emph{enlargement} holds.

Let $\G$ be a family of sets of cardinality $|\G|\le\mathfrak{c}$
with the finite intersection property. Since $|\PP(\N)|=\mathfrak{c}$,
there exists an onto function $\psi:\PP(\N)\twoheadrightarrow\G$.
By the finite intersection property of $\G$,
for every $i\in I$ we can pick an element
$a_i\in\bigcap_{A\in i}\psi(A)$. Then let
$$\nu:=[(a_i\mid i\in I)]_\U\in\hN$$
be the element of the ultrapower given by the $\U$-equivalence class 
of the sequence $(a_i\mid i\in I)$. 
We claim that $\nu\in\bigcap_{G\in\G}{}^*G$.
To see this, notice first that since $\psi$ is onto, 
we have $\bigcap_{G\in\G}{}^*G=\bigcap_{A\subseteq\N}{}^*(\psi(A))$.
Recall that ${}^*(\psi(A))=[(\psi(A)\mid i\in I)]_\U$ is the $\U$-equivalence
class of the constant sequence with value $\psi(A)$.
By definition, for every $A\subseteq\N$ and for every $i\in\check{A}$, we have $A\in i$,
and so $a_i\in \psi(A)$. This shows that 
$\check{A}\subseteq\{i\in I\mid a_i\in \psi(A)\}\in\U$, and hence
$\nu\in{}^*(\psi(A))$.

\smallskip
To give a complete proof of the \emph{transfer principle}
one would need to correctly formalize its statement, and this requires
the formal notion of first-order formula. 
Since going into specific logical aspects is outside the scope of this paper,
we content ourselves to give a flavor of its content and so,
as an example, we will only give here a detailed proof that the
ultrapower $\hN=\N^I\!/\U$ is the positive part of a discretely ordered ring.\footnote
{The interested reader in a precise formulation of the \emph{transfer} principle
and its formal proof is refereed to \cite[Chapter 2]{dgl}, and references therein.}

For simplicity, in the following we write $[f]$ instead of $[f]_\U$ to
denote the $\U$-equivalence of a function $f:I\to\N$.
Besides, we identify each $n\in\N$ with the $\U$-equivalence class of the constant
sequence $c_n:i\mapsto n$ with value $n$, so that $\N\subseteq\hN$.

According to Definition \ref{def-extension}, the sum $\boldsymbol{+}$, product $\boldsymbol{\cdot}$, 
and order $\boldsymbol{<}$ on the ultrapower $\hN=\N^I\!/\U$ are defined as follows, respectively:\footnote
{For simplicity, we used boldface symbols for the sum, product, and order, instead of putting stars.}
\begin{itemize}
\item
$[f]\boldsymbol{+}[g]:=[f+g]$ where $f+g:i\mapsto f(i)+g(i)$ is 
the sum of the functions $f$ and $g$ defined point-wise.
\item
$[f]\boldsymbol{\cdot}[g]:=[f\cdot g]$ where $f\cdot g:i\mapsto f(i)\cdot g(i)$ is the 
product of the functions $f$ and $g$ defined point-wise.
\item
$[f]\boldsymbol{<}[g]\Leftrightarrow\{i\in I\mid f(i)<g(i)\}\in\U$.
\end{itemize}

Commutativity, associativity and distributivity of the operations directly follow
from the above definitions and the fact that they hold ``point-wise" in $\N$. 
As an example, let us see
associativity of the sum in detail. 
Let $[f],[g],[h]\in\hN$ be the elements in the ultrapower that are the
$\U$-equivalence classes of the functions $f,g,h:I\to\N$, respectively. Then
$([f]\boldsymbol{+}[g])\boldsymbol{+}[h]=[f+g]\boldsymbol{+}[h]=[(f+g)+h]=[f+(g+h)]=
[f]\boldsymbol{+}[g+h]=[f]\boldsymbol{+}([g]\boldsymbol{+}[h])$.

Let us now see that the relation $\boldsymbol{<}$ is a linear order on $\hN$.
Note that $\boldsymbol{<}$ is irreflexive; in fact for every $[f]\in\hN$, the relation
$[f]\boldsymbol{<}[f]$ does not hold since $\{i\in I\mid f(i)<f(i)\}=\emptyset\notin\U$.
As for transitivity, assume that $[f]\boldsymbol{<}[g]$ and $[g]\boldsymbol{<}[h]$,
\emph{i.e.}, $\Lambda:=\{i\in I\mid f(i)<g(i)\}\in\U$ and $\Gamma:=\{i\in I\mid g(i)<h(i)\}\in\U$.
Then $[f]<[h]$ because $\{i\in I\mid f(i)<h(i)\}\supseteq\Lambda\cap\Gamma\in\U$.
Finally, given $[f],[g]\in\hN$, let $A_1:=\{i\in I\mid f(i)<g(i)\}$,
$A_2:=\{i\in I\mid f(i)=g(i)\}$, and $A_3=\{i\in I\mid g(i)<f(i)\}$.
Then we have the partition $I=A_1\cup A_2\cup A_3\in\U$, and hence one
and only one of the sets $A_i\in\U$, \emph{i.e.},
one and only one of the relations $[f]\boldsymbol{<}[g]$,
$[f]=[g]$, or $[g]\boldsymbol{<}[f]$ holds.

The compatibility of the order relation with the operations
directly follows from the corresponding ``point-wise" property.

The fact that $\hN$ is additively cancellative, \emph{i.e.},
if $[f]\boldsymbol{+}[g]=[f]\boldsymbol{+}[g']$ then $[g]=[g']$
is straightforward by noticing that
$f(i)+g(i)=f(i)+g'(i)\Leftrightarrow g(i)=g'(i)$ for every $i\in I$.
To complete the proof that $\hN$ is the positive part of an ordered ring,
we are left to show that if $[f]\boldsymbol{<}[g]$ then there exists $[h]$
such that $[f]\boldsymbol{+}[h]=[g]$. By the hypothesis,
$\Lambda:=\{i\in I\mid f(i)<g(i)\}\in\U$. Define $h:I\to\N$ by setting
$h(i)=g(i)-f(i)$ if $i\in\Lambda$, and $h(i)=1$ otherwise.
Then $[f]\boldsymbol{+}[h]=[f+h]=[g]$ since the set
$\{i\in I\mid f(i)+h(i)=g(i)\}=\Lambda\in\U$.
\end{proof}

%
%
%

\end{document}